\documentclass[review,onefignum,onetabnum]{siamart190516}

\usepackage{lipsum}
\usepackage{amsfonts}
\usepackage{graphicx}
\usepackage{epstopdf}
\usepackage{algorithmic}
\usepackage{amssymb}
\ifpdf
  \DeclareGraphicsExtensions{.eps,.pdf,.png,.jpg}
\else
  \DeclareGraphicsExtensions{.eps}
\fi

\newsiamremark{remark}{Remark}
\newsiamremark{hypothesis}{Hypothesis}
\crefname{hypothesis}{Hypothesis}{Hypotheses}
\newsiamthm{claim}{Claim}

\headers{Implicit--explicit VBDF-2 scheme for gradient flow}{D. Hou and Z. Qiao}

\title{An implicit--explicit second order  BDF numerical scheme with variable steps for gradient flows\thanks{
Submitted to the editors DATE.
\funding{The work of D. Hou is supported by NSFC grant 12001248, the NSF of the Jiangsu Higher Education Institutions of China grant
BK20201020, the NSF of Universities in Jiangsu Province of China grant 20KJB110013 and the Hong Kong Polytechnic University grant 1-W00D. Z. Qiao  has received support from the Hong Kong Research Council RFS grant RFS2021-5S03 and GRF grant 15302919, the Hong Kong Polytechnic University grant 4-ZZLS, and the CAS AMSS-PolyU Joint Laboratory of Applied Mathematics.}}}

\author{Dianming Hou\thanks{School of Mathematics and Statistics, Jiangsu Normal University, 221116 Xuzhou, China. Current address: Department of Applied Mathematics, The Hong Kong Polytechnic University, Hung Hom, Kowloon, Hong Kong.
  (\email{dmhou@stu.xmu.edu.cn}).}
\and Zhonghua Qiao\thanks{Corresponding author. Department of Applied Mathematics, The Hong Kong Polytechnic University, Hung Hom, Kowloon, Hong Kong.
  (\email{zqiao@ployu.edu.hk}).}
}

\usepackage{amsopn}

\ifpdf
\hypersetup{
  pdftitle={An implicit--explicit second order  BDF numerical scheme with variable steps for gradient flows},
  pdfauthor={Dianming Hou, and Zhonghua Qiao}
}
\fi

\begin{document}
\graphicspath{{figures/},}
\maketitle

\begin{abstract}
  In this paper, we propose and analyze an efficient implicit--explicit (IMEX) second order in time backward differentiation formulation (BDF2) scheme with variable time steps for gradient flow problems using the scalar auxiliary variable (SAV) approach.
We prove the unconditional energy stability of the scheme for a modified discrete energy with the adjacent time step ratio $\gamma_{n+1}:=\Dt_{n+1}/\Dt_{n}\leq 4.8645$. The uniform $H^{2}$ bound for the numerical solution is derived under a mild regularity restriction on the initial condition, that is  $\phi(\x,0)\in H^{2}$. Based on this uniform bound, a rigorous error estimate of the numerical solution is carried out on the temporal nonuniform  mesh.
Finally, serval numerical tests are provided to validate the theoretical claims.  With the application of an adaptive time-stepping strategy, the efficiency of our proposed scheme can be clearly observed in the coarsening dynamics simulation.
\end{abstract}



\section{Introduction}
Many science and engineering problems can be modeled by partial differential equations (PDE) having the structure of gradient flows.
The evolution PDE systems are resulted from the energetic variational principle of the
total free energy in different Sobolev spaces.
Typical gradient flows include the Allen-Cahn and Cahn-Hilliard equations. 
The general form of gradient flow can be written as
\be\label{prob0}
\dps\frac{\partial \phi}{\partial t}=-\mbox{grad}_{H} E(\phi),
\ee
where $E[\phi(\x,t)]$ is the free energy functional associated to the physical problem,
 and $\mbox{grad}_{H} E(\phi)$ is the variational derivative of $E$ in the Sobolev space $H$.
Then the above PDE system satisfies the following energy dissipation law
 \bq\label{EDlaw}
 \frac{d}{dt}E(\phi)
 =\Big(\mbox{grad}_{H} E(\phi), \frac{\partial \phi}{\partial t}\Big)
 =-\|\mbox{grad}_{H} E(\phi)\|_{0}^{2},
 \eq
 where $(\cdot,\cdot)$ and $\|\cdot\|_0$ are the $L^{2}-$inner product and the associated norm, respectively.
This indicates the non-increasing in time nature of the energy functional $E$ for the gradient flow system. Moreover, in many cases, it is observed that the evolution of the free energy $E(\phi)$ usually involves both fast and slow stages of change in the long time simulation. Therefore, it is highly desired to develop high order unconditionally stable numerical schemes on the temporal nonuniform mesh, in which the adaptive time-stepping statigies can be easily applied without worrying about the instability. This is essentially important for the gradient flow problems with the large scale and long time simulation. For the single step time integration method, such as the first order BDF scheme and second order Crank-Nicolson scheme, it is relatively easy to construct the unconditionally energy stable numerical schemes for gradient flows. However, it is completely different and quite difficult to numerically investigate the multi-step methods, especially on the temporal nonuniform meshes; see e.g., \cite{BJ98,CWYZ19,LZ20,RML82,TV197,WS08} for related research of BDF$k$ methods with $2\leq k\leq 6$. It is noted that the BDF2 scheme usually achieves stronger stability than the second order Crank-Nicolson based schemes, particularly for the problems with strong stiffness \cite{DYL18,HAX19,HX21_3}.

Our goal in this paper is to construct and analyze the linear unconditionally energy stable and variable time-stepping BDF2 scheme for \eqref{prob0} with a mild restriction on the adjacent time step ratio.
In \cite{CWWW19,YCWW18}, the second order BDF numerical schemes have been developed for the Cahn-Hilliard equations with the double well and logarithmic nonlinear potentials, respectively. The energy stability and the convergence analysis were also derived on the temporal uniform mesh by using the Douglas-Dupont regularization technique. Meanwhile, it has attracted increasing attention and has been extensively studied in constructing robust and efficient BDF2 schemes for parabolic PDEs on the temporal nonuniform mesh, in which the adaptive time-stepping strategies can be easily applied to improve the efficiency of the numerical simulation. Becker \cite{BJ98} numerically investigated the variable time-stepping BDF2 scheme for linear parabolic PDEs, in which the rigorous stability and convergence analysis were derived with a restriction on the adjacent time step ratio $\gamma_{n+1}<(2+\sqrt{13})/3\approx1.8685.$ However, an exponential prefactor was involved in their numerical analysis, which could blow up for certain choice of time-step series at vanishing time-steps.
Chen et al. presented a nonuniform BDF2 scheme with convex splitting method for the Cahn-Hilliard equation in \cite{CWYZ19}.
A new developed Gr\"onwall inequality was employed to remove the blowing-up prefactor in their convergence analysis under the restriction on the adjacent time step ratio $\gamma_{n}\leq1.534$.
In \cite{LTZ20}, Liao, et al. investigated the BDF2 scheme for the Allen-Cahn equation on the temporal nonuniform mesh, in which the nonlinear potential term was treated implicitly. The unique solvability and energy stability of their proposed scheme were derived with the time step ratio $\gamma_{n}\leq 3.561$ and a mild restriction on the time step sizes. Moreover, if $\gamma_{n}\leq 1+\sqrt{2}$, the discrete maximum bound principle of the numerical scheme has been obtained with the kernel recombination technique.
With this uniform bound of the numerical solutions, they have presented a rigorous error analysis of the fully implicit scheme by using a novel discrete Gr\"onwall inequality.
Recently, the BDF2 scheme with variable
time steps has been investigated for molecular beam epitaxial (MBE) model without slope selection in \cite{TZZ21}. Under the constrain condition on the adjacent time step ratio $\gamma_{n}\leq 4.8465,$
it has been proved that the proposed numerical scheme is unconditionally energy stable and the numerical solution is of second-order convergence in time. Since the existing variable BDF2 schemes for gradient flows treat the nonlinear terms fully or partially implicitly, nonlinear iterations must be conducted at each time step.

As we know, the nonlinear terms involved in gradient flows usually yield a severe stability restriction on the time step sizes and the ratio of the adjacent time steps in the numerical schemes. Many efforts have been devoted to remove this restriction by carefully dealing with the nonlinear terms in the gradient flow problems. These includes, but not limited to, the convex splitting method  \cite{BLWW13,Ell93,Eyr98,GLWW14}, the linear stabilized method \cite{DJLQ_rev,JLQZ18,LQ17,LQT16,SY10,Xu06}, the invariant energy quadratization (IEQ) method \cite{Yan16,Zha17} and the scalar auxiliary variable (SAV) approach \cite{Shen17_1,Shen17_2}. Particularly, the SAV method achieves more convenience to design the efficient and unconditionally energy stable schemes for the gradient flow problems, where a modified energy is defined involving an scalar auxiliary variable. Very recently, Huang et al. \cite{HS21_1,HSY21} presented an implicit-explicit linear and unconditionally energy stable BDF$k$ $(1\leq k\leq5)$ scheme for gradient flows using the SAV method. However, its unconditional energy stability only involved the auxiliary variable without the information of the phase function. The convergence analysis was also derived, but only available on the temporal unform mesh \cite{HS21_1}.

In this paper, we propose and analyze an efficient implicit--explicit BDF2 scheme with variable time steps for gradient flow problems using the SAV approach.
We prove the unconditional energy stability of the scheme for a modified discrete energy with the adjacent time step ratio $\gamma_{n+1}:=\Dt_{n+1}/\Dt_{n}\leq 4.8645$.
The main contributions are as follows. First, we develop a linear unconditionally energy stable BDF2 scheme on the temporal nonuniform mesh for the first time for gradient flows under the mild restriction on the adjacent time step ration $\gamma_{n}\leq 4.8645$. Moreover, the stability of the numerical scheme indicates the uniform $H^{1}$ bound for the numerical solutions, which plays an important role in deriving the error estimates.
Secondly, the rigorous error analysis has been established with a proper regularity assumption on the solution.

The rest of this paper is organized as follows: In \cref{sec:sect2}, we construct the IMEX BDF2 scheme for gradient flows on the nonuniform grid mesh using the SAV approach, and
carry out the stability results for the proposed schemes. Based on the stability form of the numerical scheme, the uniform $H^{2}$ bound for the numerical solutions is investigated in \cref{sec:sect3}.
In \cref{sec:sect4}, a rigorous convergence analysis is derived for both $L^{2}$ and $H^{-1}$ gradient flows.
Several numerical examples are given in \cref{sec:sect5} to validate the theoretical prediction of the proposed method.
Finally, the paper ends with some concluding remarks.


\section {Scalar auxiliary variable approach}\label{sec:sect2}

We first introduce some notations which are used throughout the paper. Let $H^{m}(\Omega)$ and $\|\cdot\|_{m},
m=0,\pm1,\cdots$, denote the standard Sobolev spaces and their norms, respectively. In particular, the norm and inner product of $L^{2}(\Omega):=H^{0}(\Omega)$ are denoted by $\|\cdot\|$ and $(\cdot,\cdot)$ respectively. Here, $\Omega\subset \mathbb{R}^{d}, d=1,2,3$ is a bounded domain with smooth boundary.

In what follows, we focus our attention on a typical energy functional $E(\phi)$, defined by
\bq
E(\phi)=\int_{\Omega}\Big(\frac{\varepsilon^{2}}{2}|\nabla\phi|^{2}+F(\phi)\Big)d\x,
\eq
where $F(\phi)$ represents the general nonlinear potential function.
 Taking the Sobolev space $H$ to be
$L^2(\Omega)$ and $H^{-1}(\Omega)$ in \eqref{prob0} respectively, the gradient flows read
\brr\label{prob}
\begin{cases}
\dps\frac{\partial \phi}{\partial t}
=G_H \mu,\\[9pt]
\dps\mu=-\varepsilon^{2}\Delta\phi+F'(\phi),
\end{cases}
\err
where
\bq\label{ef}
G_H:=
\begin{cases}
\begin{array}{l}
-I,\quad \mbox{ for $H:=L^{2}$},\\[9pt]
\Delta,\quad \mbox{ for $H:=H^{-1}$.}
\end{array}
\end{cases}
 \eq
 For the sake of simplicity we consider the periodic boundary condition or homogeneous Neumann boundary condition for \eqref{prob}, i.e.,
\bq\label{N_B}
\dps\frac{\partial\phi}{\partial \n}\big|_{\partial\Omega}=0 \mbox{ for } G_H:=-I
    \mbox{ and }\frac{\partial\phi}{\partial \n}\big|_{\partial\Omega}=\frac{\partial\Delta\phi}{\partial \n}\big|_{\partial\Omega}=0 \mbox{ for } G_H:=\Delta,
\eq
where $\n$ is the unit outward normal vector on the boundary.

The key idea of the SAV approach \cite{Shen17_1} is to introduce an scalar auxiliary function, defined by
\bq\label{SAV}
r(t)=\sqrt{E_{1}(\phi)}:=\sqrt{\int_{\Omega}g(\phi)d\x+C_{0}},
\eq
where $g(\phi)=F(\phi)-\frac{\lambda}{2}|\phi|^{2}$ with $\lambda\geq0$, and $C_{0}$ is a positive constant such that $E_{1}(\phi)>0.$
Then, the original system \eqref{prob} can be equivalently rewritten as:
 \brr\label{re_prob}
 \begin{cases}
 &\dps\frac{\partial \phi}{\partial t}=G_H\mu,\\[9pt]
 &\dps\mu=-\varepsilon^{2}\Delta\phi+\lambda\phi+\frac{r(t)}{\sqrt{E_{1}(\phi)}}V(\xi)g'(\phi),\\[15pt]
 &\dps\frac{d r}{d t}=\frac{V(\xi)}{2\sqrt{E_{1}(\phi)}}\int_{\Omega}g'(\phi)
\frac{\partial \phi}{\partial t}d\x,\ \ \ \xi=\frac{r(t)}{\sqrt{E_{1}(\phi)}},
 \end{cases}
 \err
 where $V(\xi)$ is a real function of $\xi$ such that $V(1)\equiv1$ and $V(\xi)\in C^{2}(\mathbb{R}).$
 Taking $L^2-$inner products of the above equations with $\mu, \frac{\partial \phi}{\partial t}$,
 and $2r(t)$ respectively gives
 \bq\label{mod_energ}
 \dps\frac{d(\frac{\varepsilon^{2}}{2}\|\nabla\phi\|^{2}+\frac{\lambda}{2}\|\phi\|^{2}+r^{2})}{dt}=(G_H\mu,\mu)\leq0.
 \eq
 Noticing that
 $$\dps\frac{d(\frac{\varepsilon^{2}}{2}\|\nabla\phi\|^{2}+\frac{\lambda}{2}\|\phi\|^{2}+r^{2})}{dt}=\frac{d(E(\phi)+C_{0})}{dt}=\frac{d}{dt}E(\phi),
 $$
 we obtain the same energy dissipation law as \eqref{EDlaw} at the continuous level. However, we will see the constructed numerical scheme in this paper preserving a modified free energy dissipation more like the form of \eqref{mod_energ}.
{\color{black}
\begin{remark}
For the special case $V(\xi)\equiv1$, the new proposed SAV approach \eqref{re_prob} comes to be the conventional SAV method reported in \cite{Shen17_1}.
\end{remark}
}
Starting with \eqref{re_prob}, we are now ready to construct various time stepping
 schemes to approximate the solution $\phi$ with variable time steps.  Let $\Dt_{n}:=t_{n}-t_{n-1}, n=0,1,\dots, N$ be the time step size, $\gamma^{n+1}=\Dt_{n+1}/\Dt_{n}$ be the adjacent time--step ratios,
 and $\tau=\max\{\Dt_{n},n=1,2,\cdots\}$ be the maximum time step size of the temporal mesh.

\subsection{A first order scheme}

 A combination of some first order discretizations to each term of \eqref{re_prob} gives
 the following scheme:
 \begin{subequations}\label{eq1}
 \begin{align}
 &\dps{}\frac{\phi^{n+1}-\phi^{n}}{\Dt_{n+1}}=G_H\mu^{n+1},\label{eq1_1}\\
 &\dps\mu^{n+1}=-\varepsilon^{2}\Delta\phi^{n+1}+\lambda\phi^{n+1}+\frac{r^{n+1}}{\sqrt{E_{1}^{n}}}V(\xi^{n+1})g'(\phi^{n}),\label{eq1_2}\\
 &\dps\frac{r^{n+1}-r^{n}}{\Dt_{n+1}}=\frac{V(\xi^{n+1})}{2\sqrt{E_{1}^{n}}}\int_{\Omega}g'(\phi^{n})
\frac{\phi^{n+1}-\phi^{n}}{\Dt_{n+1}}d\x,\ \ \ \xi^{n+1}=\frac{r^{n+1}}{\sqrt{E_{1}^{n}}},\label{eq1_3}
 \end{align}
 \end{subequations}
where $\phi^{n}$ is an approximation to $\phi(t^n)$ and $E^{n}_{1}:=E_{1}(\phi^{n})$.

Similar to the stability analysis of the traditional SAV method \cite{Shen17_1}, it is easy to check that the scheme \eqref{eq1} is unconditionally energy stable in the sense that the following discrete energy law holds
\brr\label{eq2}
&\dps\Big[\frac{\varepsilon^{2}}{2}\|\nabla\phi^{n+1}\|^{2}+\frac{\lambda}{2}\|\phi^{n+1}\|^{2}+|r^{n+1}|^{2}\Big]-\Big[\frac{\varepsilon^{2}}{2}\|\nabla\phi^{n}\|^{2}+\frac{\lambda}{2}\|\phi^{n}\|^{2}+|r^{n}|^{2}\Big]\\[7pt]
\leq&\dps(G_H\mu^{n+1},\mu^{n+1})\Dt_{n+1}\leq0.
\err

\subsection{An IMEX second-order BDF scheme}
We first recall the quadratic interpolation operator $\Pi_{2,j}\phi(t)$: for $\phi(t), t\in[t_{j-1},t_{j+1}]$,  $\Pi_{2,j}\phi(t)$ interpolates
$\phi(t)$ at three points
$(t_{j-1},\phi(t_{j-1})), (t_{j},\phi(t_{j}))$ and $(t_{j+1},\phi(t_{j+1}))$. That is, for $j=1,2\cdots,N-1$,
\beq
\Pi_{2,j}\phi=\phi(t_{j-1})l_{j-1}(t)+\phi(t_{j})l_{j}(t)+\phi(t_{j+1})l_{j+1}(t),
\eeq
where $l_{j}(t)$ is the basis of Lagrange interpolation polynomial, defined by
$$\dps l_{k}(t):=\prod_{i=j-1,\atop i\neq k}^{j+1}\frac{t-t_{i}}{t_{k}-t_{i}}$$
for $k=j-1, j$, and $j+1$.
Moreover, we have
\bry
\partial_{t}(\Pi_{2,j}\phi(t))\Big|_{t=t_{j+\sigma}}=&\dps\frac{1}{\Dt_{j+1}}\Big[\frac{1+2\sigma\gamma_{j+1}}{1+\gamma_{j+1}}\phi(t_{j+1})-(1+(2\sigma-1)\gamma_{j+1})\phi(t_{j})\\[8pt]
&\dps+\frac{(2\sigma-1)\gamma^{2}_{j+1}}{1+\gamma_{j+1}}\phi(t_{j-1})\Big]
\ery
for $t_{j+\sigma}:=t_{n}+\sigma\Dt_{n+1}$ with $\frac{1}{2}\leq\sigma\leq1.$
Then, the second order finite differentiation formula $F^{j+\sigma}_{2}\phi$ to approximate $\phi'(t)$ at $t=t_{j+\sigma}$ reads
\bq\label{BDF_s}
\dps F^{j+\sigma}_{2}\phi:=\frac{1}{\Dt_{j+1}}\Big[\frac{1+2\sigma\gamma_{j+1}}{1+\gamma_{j+1}}\phi^{j+1}-(1+(2\sigma-1)\gamma_{j+1})\phi^{j}+\frac{(2\sigma-1)\gamma^{2}_{j+1}}{1+\gamma_{j+1}}\phi^{j-1}\Big],
\eq
where $\phi^{j}$ is an approximation to $\phi(t_{j}).$
Note that when $\sigma=1$, it becomes the classical second-order BDF with variable time steps:
\be\label{BDF2b}
F^{j+1}_{2}\phi=\frac{1}{\Dt_{j+1}}\Big[\frac{1+2\gamma_{j+1}}{1+\gamma_{j+1}}\phi^{j+1}-(1+\gamma_{j+1})\phi^{j}+\frac{\gamma^{2}_{j+1}}{1+\gamma_{j+1}}\phi^{j-1}\Big],
\ee
 and Crank-Nicolson formula
\bex
F^{j+1/2}_{2}\phi=\frac{\phi^{j+1}-\phi^{j}}{\Dt_{j+1}}
\eex
for $\sigma=\frac{1}{2}$.

Now, we are ready to construct an IMEX second-order BDF scheme with variable time steps, called VBDF2 hereafter, for \eqref{re_prob}:
 \begin{subequations}\label{BDF_2}
 \begin{align}
 &\dps F^{n+\sigma}_{2}\phi=G_H\mu^{n+\sigma},\label{BDF_2_1}\\
 &\dps\mu^{n+\sigma}=-\varepsilon^{2}\Delta\phi^{n+\sigma}+\lambda\phi^{n+\sigma}+\frac{r^{n+1}}{\sqrt{E_{1}^{n}}}V(\xi^{n+1})g'(\phi^{*,n+\sigma}),\label{BDF_2_2}\\
 &\dps\frac{r^{n+1}-r^{n}}{\Dt_{n+1}}=\frac{V(\xi^{n+1})}{2\sqrt{E_{1}^{n}}}\int_{\Omega}g'(\phi^{*,n+\sigma})
\frac{\phi^{n+1}-\phi^{n}}{\Dt_{n+1}}d\x, \ \  \xi^{n+1}=\frac{r^{n+1}}{\sqrt{E_{1}^{n}}},\label{BDF_2_3}
 \end{align}
 \end{subequations}
 where $\phi^{n+\sigma}:=\sigma\phi^{n+1}+(1-\sigma)\phi^{n}$ and $\phi^{*,n+\sigma}:=\phi^{n}+\sigma\gamma_{n+1}(\phi^{n}-\phi^{n-1})$. Here we set $V(\xi)$ to be a functional of $\xi$ such that
 \bq\label{chv}
 \dps V(1)=1, \ \ \lim_{\xi\rightarrow1}\frac{V(\xi)-1}{1-\xi}=1,
 \eq
 that is,  $V'(1)=-1.$ From \eqref{BDF_2_3}, it is obvious that the numerical approximations to $R(t)$ and $\xi(t)$ are only first order accurate in time, viz.,
 \bq\label{ap_xi}
 R^{n+1}=R(t_{n+1})+C_{1}\Dt_{n+1},\ \ \xi^{n+1}=\xi(t_{n+1})+C_{2}\Dt_{n+1}=1+C_{2}\Dt_{n+1}.
 \eq
 It seems that the proposed VBDF-2 scheme \eqref{BDF_2} for the evaluation of $\phi$ is only of first order.
In fact, in the SAV approach \eqref{re_prob},
the auxiliary variable functional $R(t)$ and $\xi(t)$ affect the phase function $\phi$ of the original problem \eqref{prob} only by the term $\frac{R(t)}{\sqrt{E_{1}(\phi)}+C_{0}}V(\xi)$ in the discrete level. Moreover, combining with \eqref{chv} and \eqref{ap_xi}, it is easy to check that the discrete term $\frac{r^{n+1}}{\sqrt{E_{1}^{n}}}V(\xi^{n+1})=\xi^{n+1}V(\xi^{n+1})$ is of second order approximation to $1$. Therefore, the VBDF2 scheme is actually a second-order numerical method of \eqref{re_prob}. This will be rigorously verified by the theoretical analysis to be presented in Section 4.
Before rigorously proving the energy stability of VBDF2 scheme \eqref{BDF_2}, we
need the following inequality, seeing also Lemma 3.1 in \cite{HX21_1} with the parameter $\alpha=1$,
\brr\label{ideq1}
\dps (\phi^{n+1}-\phi^{n}) \ F^{n+\sigma}_{2}\phi\geq\dps G_{\sigma}^{n+1}
-G_{\sigma}^{n}+G(\gamma_{n+1},\gamma_{n+2})\frac{|\phi^{n+1}-\phi^{n}|^{2}}{2\Dt_{n+1}},
\err
where $G_{\sigma}^{n}:=\frac{(2\sigma-1)\gamma^{3/2}_{n+1}}{2(1+\gamma_{n})}\frac{|\phi^{n}-\phi^{n-1}|^{2}}{\Dt_{n}}$ and $G(s,z):=\frac{2+4\sigma s-(2\sigma-1)s^{\frac{3}{2}}}{1+s}-\frac{(2\sigma-1)z^{\frac{3}{2}}}{1+z}.$
$G(s,z)$ is increasing in $(0,s_{*}(\sigma))$ and decreasing in $(s_{*}(\sigma),+\infty)$ with respect to $s$, and decreasing for the variable $z\in(0,+\infty)$.
Moreover, there exists $0<z_{*}<4$ such that $G(z,z)$ is increasing in $(0,z_{*})$ and decreasing in $(z_{*},+\infty)$, since $G(0,0)=2\geq G(4,4)=2-\frac{8(2\sigma-1)}{5}\geq\frac{2}{5}.$
Let $\gamma_{**}(\sigma)$ be the positive root of $G(z,z)=0$ and $4\leq\gamma_{*}(\sigma)<\gamma_{**}(\sigma)$.
Thus, it follows from $G(0,z)=G(4,z)$ that
\bry
G(s,z)&\dps\geq \min\{G(0,\gamma_{*}(\sigma)),G(\gamma_{*}(\sigma),\gamma_{*}(\sigma))\}\geq G(\gamma_{*}(\sigma),\gamma_{*}(\sigma))\\
&\dps>G(\gamma_{**}(\sigma),\gamma_{**}(\sigma))=0
\ery
for any $0<s,z\leq\gamma_{*}(\sigma)<\gamma_{**}(\sigma)$.
 It is easy to check that the root function $\gamma_{**}(\sigma)$ is decreasing for $\sigma\in[\frac{1}{2},1]$ with $\gamma_{**}(1)\approx4.8645$ and $\gamma_{**}(\sigma)\rightarrow+\infty,$ as $\sigma\rightarrow\frac{1}{2}.$ Thus $\gamma_{**}(\sigma)\geq 4.8645$ for any $\sigma\in[\frac{1}{2},1].$

The stability property of the VBDF-2 scheme \eqref{BDF_2} is stated in the following theorem.
\begin{theorem}\label{st_the}
For $1/2\leq\sigma\leq1$ and $0<\gamma_{n+1}\leq\gamma_{**}(\sigma), n=1,\cdots,M-1$,
 the VBDF-2 scheme \eqref{BDF_2} is unconditionally energy stable in the sense that
\bq\label{eq2_1}
\dps E^{n+1}_{H}-E^{n}_{H}\leq0,
\eq
where the discrete modified energy $E^{n}_{H}$ is defined by:
\beq
\dps E^{n}_{H}:=
\begin{cases}
\begin{array}{l}
\dps\frac{(2\sigma-1)\gamma^{\frac{3}{2}}_{n+2}}{2+2\gamma_{n+2}}\frac{\|\phi^{n}-\phi^{n-1}\|^{2}}{\tau_{n}}+
E^{n}_{Mod}, \mbox{ for $H:=L^{2}$,}\\
\dps\frac{(2\sigma-1)\gamma^{\frac{3}{2}}_{n+1}}{2+2\gamma_{n+1}}\frac{\|\nabla^{-1}(\phi^{n}-\phi^{n-1})\|^{2}}{\tau_{n}}+
E^{n}_{Mod}, \mbox{ for $H:=H^{-1}$,}
\end{array}
\end{cases}
\eeq
with $E^{n}_{Mod}=\frac{\varepsilon^{2}}{2}\|\nabla\phi^{n}\|^{2}+\frac{\lambda}{2}\|\phi^{n}\|^{2}+|r^{n}|^{2}$.
\end{theorem}
\begin{proof}
We deduce from taking the inner products of \eqref{BDF_2_1}, \eqref{BDF_2_2}, and \eqref{BDF_2_3} with $-G_{H}^{-1}(\phi^{n+1}-\phi^{n}),$ $\phi^{n+1}-\phi^{n}$, and $2r^{n+1}$, respectively
\bry
\dps\Big(F^{n+\sigma}_{2}\phi,-G^{-1}_{H}(\phi^{n+1}-\phi^{n})\Big)=&\dps-(\mu^{n+1},\phi^{n+1}-\phi^{n}),\\[8pt]
\dps\Big(\mu^{n+1},\phi^{n+1}-\phi^{n}\Big)
 = &\dps\varepsilon^{2}\big(\nabla\phi^{n+\sigma},\nabla(\phi^{n+1}-\phi^{n})\big)+\lambda\big(\phi^{n+\sigma},\phi^{n+1}-\phi^{n}\big)\\[8pt]
 &\dps+\frac{r^{n+1}V(\xi^{n+1})}{\sqrt{E_{1}^{n}}}\Big(g'(\phi^{*,n+1}),\phi^{n+1}-\phi^{n}\Big),\\[17pt]
 \dps\frac{2r^{n+1}(r^{n+1}-r^{n})}{\Dt_{n+1}}=&\dps\frac{r^{n+1}V(\xi^{n+1})}{\sqrt{E_{1}^{n}}}\int_{\Omega}g'(\phi^{*,n+1})
\frac{\phi^{n+1}-\phi^{n}}{\Dt_{n+1}}d\x.
\ery
It follows from the above equations that
\brr\label{equat2}
&\dps\Big(F^{n+\sigma}_{2}\phi,-G^{-1}_{H}(\phi^{n+1}-\phi^{n})\Big)+\varepsilon^{2}\big(\nabla\phi^{n+\sigma},\nabla(\phi^{n+1}-\phi^{n})\big)\\[8pt]
&\dps+\lambda\big(\phi^{n+\sigma},\phi^{n+1}-\phi^{n}\big)+2r^{n+1}(r^{n+1}-r^{n})=0.
\err
Furthermore, applying the inequality \eqref{ideq1} and the identities
 \brr\label{Id}
2(\sigma a^{k+1}+(1-\sigma)a^{k})(a^{k+1}-a^{k})
=&|a^{k+1}|^{2}-|a^{k}|^{2}+(2\sigma-1)|a^{k}-a^{k-1}|^{2},\\[9pt]
2a^{k+1}(a^{k+1}-a^{k})
=&|a^{k+1}|^{2}-|a^{k}|^{2}+|a^{k+1}-a^{k}|^{2}
\err
to \eqref{equat2} and dropping some uninfluential terms,
the modified energy law \eqref{eq2_1} can be easy derived.
\end{proof}

Next, we will show how to efficiently solve the VBDF2 scheme \eqref{BDF_2}. It follows from \eqref{BDF_s} and \eqref{BDF_2} that
\bry
&\dps\Big(\frac{1+2\sigma\gamma_{n+1}}{\Dt_{n+1}(1+\gamma_{n+1})}-\sigma G_H(-\varepsilon^{2}\Delta+\lambda)\Big)\phi^{n+1}-\xi^{n+1}V(\xi^{n+1})G_H g'(\phi^{*,n+1})\\
=&\dps\frac{1}{\Dt_{n+1}}\Big((1+(2\sigma-1)\gamma_{n+1})\phi^{n}-\frac{(2\sigma-1)\gamma_{n+1}^{2}}{1+\gamma_{n+1}}\phi^{n-1}\Big)
+(1-\sigma)\varepsilon^{2}G_H\Delta\phi^{n}.
\ery
We denote the right hand side of the above equation by $h(\phi^{n},\phi^{n-1}).$
Then it holds that
\bq\label{phi_epx}
\phi^{n+1}=\phi^{n+1}_{1}+\xi^{n+1}V(\xi^{n+1})\phi^{n+1}_{2}
\eq
with $\phi^{n+1}_{1}$ and $\phi^{n+1}_{2}$ being solved respectively by
\begin{subequations}\label{semi_D}
 \begin{align}
 &\dps\Big(\frac{1+2\sigma\gamma_{n+1}}{\Dt_{n+1}(1+\gamma_{n+1})}-\sigma G_H(-\varepsilon^{2}\Delta+\lambda)\Big)\phi^{n+1}_{1}=h(\phi^{n},\phi^{n-1}),\label{phi_1}\\
 &\dps\Big(\frac{1+2\sigma\gamma_{n+1}}{\Dt_{n+1}(1+\gamma_{n+1})}-\sigma G_H(-\varepsilon^{2}\Delta+\lambda)\Big)\phi^{n+1}_{2}=G_{H} g'(\phi^{*,n+1}).\label{phi_2}
     \end{align}
 \end{subequations}
 Actually, it is easy to see $\phi^{n+1}_{1}+\phi^{n+1}_{2}$ is the numerical solution of $\phi(t_{n+1})$ with the traditional second order stabilized method for the original gradient flow \eqref{prob}. In order to numerically solve the auxiliary function $\xi^{n+1}$,
we substitute equation \eqref{phi_epx} into \eqref{BDF_2_3} to obtain
\brr\label{xi_sol}
&\dps\xi^{n+1}\sqrt{E_{1}^{n}}-R^{n}-\frac{V(\xi^{n+1})}{\sqrt{E_{1}^{n}}}\Big[\xi^{n+1}V(\xi^{n+1})(g'(\phi^{*,n+1}),\phi^{n+1}_{2})\\
&\hspace{6.5cm}\dps+(g'(\phi^{*,n+1}),\phi^{n+1}_{1}-\phi^{n})\Big]=0.
\err
Denoting the left side of the above equation by $W(\xi^{n+1})$, and combining with $V(1)=1$ and $V^{'}(1)=-1$, we have the following identities
\brr
W(1)&\dps=\sqrt{E_{1}^{n}}-R^{n}-\frac{\big(g'(\phi^{*,n+1}),\phi^{n+1}_{1}+\phi^{n+1}_{2}-\phi^{n}\big)}{\sqrt{E_{1}^{n}}}\sim O(\Dt_{n+1}),\\
W^{'}(1)&\dps=\sqrt{E_{1}^{n}}+\frac{\big(g'(\phi^{*,n+1}),\phi^{n+1}_{1}+\phi^{n+1}_{2}-\phi^{n}\big)}{\sqrt{E_{1}^{n}}}\sim \sqrt{E_{1}^{n}}+O(\Dt^{2}_{n+1}).
\err
Thus $\xi^{n+1}$ can be efficiently evaluated by solving the above nonlinear algebraic equation \eqref{xi_sol} using one step Newton's iteration with the initial condition $\xi^{0}=1$, since we only require a first order accuracy of $\xi^{n+1}$ in time with the VBDF-2 scheme \eqref{BDF_2}.

To summarize, the scheme \eqref{BDF_2} can be efficiently implemented
by following the lines:

$(i)$ Calculate $\phi^{n+1}_{1}$ and $\phi^{n+1}_{2}$ from \eqref{phi_1} and \eqref{phi_2} respectively, which can be realized in parallel.

$(ii)$ Evaluate $\xi^{n+1}$ using \eqref{xi_sol}, and then $\phi^{n+1}$ can be obtain by using \eqref{phi_epx}.

The total computational complexity at each time step is essentially of
solving two (or four) Poisson type equations with constant coefficients, since the computational cost for solving the nonlinear algebraic problem \eqref{xi_sol} by the one step Newton's iteration can be negligible comparing the cost of solving $\phi^{n+1}_{1}$ and $\phi^{n+1}_{2}$. Hence,the new proposed SAV method \eqref{re_prob} is essentially as efficient as the traditional SAV approach \cite{Shen17_1}. However, the traditional second order BDF based SAV approach in \cite{Shen17_1} is only unconditional energy stable on the uniform temporal mesh. And the energy stability of our proposed numerical scheme \eqref{BDF_2} can be preserved for more general temporal meshes with slight restrictions on the ratios of the adjacent time step sizes.
\section{$H^{2}$ bound for the numerical solutions}
\label{sec:sect3}
From the energy stability \eqref{eq2} and \eqref{eq2_1}, it follows that for all $n=1,2,\cdot,N$
\beq
\|\phi^{n}\|_{1}+|r^{n}|\leq M, \mbox{ if } \phi^{0}\in H^{1}(\Omega),
\eeq
for both of the first order scheme \eqref{eq1} and the VBDF-2 scheme \eqref{BDF_2}. Here M is a positive constant depending only on $\Omega$ and the initial condition $\phi(x,0)$. This implies the $H^{1}$ bound of the numerical solution $\phi^{n}$ and the uniform bound of the auxiliary variable $r^{n}$ for all $n=1,2,\cdots, N.$ In the following, we aim to derive the unifrom $H^{2}$ bound of the numerical solution of the above proposed schemes for gradient flows \eqref{prob}. For simplicity, we only consider the special case with $\sigma=1$ in what follows, and denote $\gamma_{*}:=\gamma_{*}(1).$ Then we have $4\leq\gamma_{*}<4.8645$. Now, we first recall some results about the solution existence and regularity of $L^{2}$ and $H^{-1}$ gradient flows.
\begin{lemma}[\cite{Temam97,SX18}]\label{Th1}
$(i)$ For the $L^{2}$-gradient flow, that is $G_{H}=-I$, assume $u^{0}\in H^{2}(\Omega)$ and
\bq\label{regu_f}
|F''(x)|<C(|x|^{p}+1), p>0 \mbox{ arbitrary if } d=1,2;~~0<p<4 \mbox{ if d=3.}
\eq
Then for any $T>0,$ the $L^{2}$-gradient flow \eqref{prob0} has a unique solution in the space
\beq
C([0,T];H^{2}(\Omega))\cap L^{2}(0,T;H^{3}(\Omega)).
\eeq\\
$(ii)$ For the $H^{-1}$-gradient flow, if $\phi^{0}\in H^{2}$, and \eqref{regu_f} and the following inequality holds
\bq\label{regu_f2}
|F'''(x)|<C(|x|^{q}+1), q>0 \mbox{ arbitrary if } d=1,2;~~0<q<3 \mbox{ if d=3.}
\eq
Then for any $T>0,$ the $H^{-1}$-gradient flow \eqref{prob0} has a unique solution in the space
\beq
C([0,T];H^{2}(\Omega))\cap L^{2}(0,T;H^{4}(\Omega)).
\eeq
\end{lemma}
\subsection{$H^{-1}$ gradient flow}

The $H^{2}$ bound of the numerical solution from VBDF2 scheme \eqref{BDF_2} for $H^{-1}$ gradient flow \eqref{prob0} is given in the following theorem.
\begin{theorem}\label{Th3}
Assume $\phi^{0}\in H^{2}$, $0<\gamma_{n+1}\leq\gamma_{*}$, and $g(\cdot)\in C^{3}(\mathbb{R})$, then the solution of the VBDF-2 scheme with $G_H=\Delta$ satisfies
\brr\label{eq_p1}
&\dps\frac{\gamma^{\frac{3}{2}}_{n+2}}{1+\gamma_{n+2}}\frac{\|\phi^{n+1}-\phi^{n}\|^{2}}{2\Dt_{n+1}}
+\frac{\varepsilon^{2}}{2}\|\Delta\phi^{n+1}\|^{2}+\frac{\lambda}{2}\|\nabla\phi^{n+1}\|^{2}\\[9pt]
\leq&\dps C(M,\gamma_{*})(1+T)+\max\Big\{\frac{\gamma_{*}^{\frac{3}{2}}}{1+\gamma_{*}},1\Big\}\cdot\frac{\varepsilon^{2}}{2}\|\Delta\phi^{0}\|^{2}.
\err
\end{theorem}

\begin{proof}
For the first time step, that is $n=0$, we use the discretization of first order scheme \eqref{eq1} with $V(\xi)\equiv1$, that is
\beq
\frac{\phi^{1}-\phi^{0}}{\Dt_{1}}+\varepsilon^{2}\Delta^{2}\phi^{1}-\lambda\Delta\phi^{1}=\frac{R^{1}}{\sqrt{E^{0}_{1}+C_{0}}}\Delta g'(\phi^{0}).
\eeq
Taking the $L^{2}$-inner product of the above equality with $\phi^{1}-\phi^{0}$ and using Young's inequality, gives
\brr\label{FS_1}
&\dps\frac{\|\phi^{1}-\phi^{0}\|^{2}}{\Dt_{1}}+\frac{\varepsilon^{2}}{2}\Big(\|\Delta\phi^{1}\|^{2}-\|\Delta\phi^{0}\|^{2}\Big)+\frac{\lambda}{2}\Big(\|\nabla\phi^{1}\|^{2}-\|\nabla\phi^{0}\|^{2}\Big)\\[9pt]
 \leq&\dps\frac{R^{1}}{\sqrt{E^{0}_{1}+C_{0}}}\Big(\Delta g'(\phi^{0}),\phi^{1}-\phi^{0}\Big)\leq C(M)\Dt_{1}\|\Delta g'(\phi^{0})\|^{2}+\frac{\|\phi^{1}-\phi^{0}\|^{2}}{2\Dt_{1}}\\
\err
Since $\phi^{0}\in H^{2}$ and $g(\cdot)\in C^{3}(\mathbb{R})$, we have $\|\phi^{0}\|_{L^{\infty}}\leq \|\phi^{0}\|_{H^{2}}\leq C$ and
\brr\label{est_g}
\|\Delta g'(\phi^{0})\|&=\|g''(\phi^{0})\Delta\phi^{0}+g'''(\phi^{0})|\nabla\phi^{0}|^{2}\|\\[7pt]
&\dps\leq \|g''(\phi^{0})\|_{L^{\infty}}\|\Delta\phi^{0}\|+\|g'''(\phi^{0})\|_{L^{\infty}}\|\nabla\phi^{0}\|^{2}_{L^{4}}
\leq C\big(1+\|\nabla\phi^{0}\|^{2}_{L^{4}}\big).
\err
By the Sobolev embedding theorem, $H^{d/4}\subset L^{4}$, and the interpolation inequality, we derive
\bq\label{eq_L4}
\|\nabla\phi^{0}\|_{L^{4}}\leq C\|\nabla\phi^{0}\|_{H^{\frac{d}{4}}}\leq C\|\nabla\phi^{0}\|^{1-\frac{d}{4}}\|\Delta\phi^{0}\|^{\frac{d}{4}}\leq C(M)\|\Delta\phi^{0}\|^{\frac{d}{4}}.
\eq
Combining with \eqref{est_g} and \eqref{eq_L4}, it gives $\|\Delta g'(\phi^{0})\|\leq C(M).$ Thus we can derive from \eqref{FS_1} that
\bq\label{eq_p}
\dps\frac{\|\phi^{1}-\phi^{0}\|^{2}}{2\Dt_{1}}+\frac{\varepsilon^{2}}{2}\|\Delta\phi^{1}\|^{2}+\frac{\lambda}{2}\|\nabla\phi^{1}\|^{2}
\leq C(M)(1+\Dt_{1})+\frac{\varepsilon^{2}}{2}\|\Delta\phi^{0}\|^{2}.
\eq
This indicates $\phi^{1}\in H^{2}$ and its $H^{2}$ bound only depends on $\phi^{0}, \Omega$ and $T$. Then we assume that $\phi^{k}\in H^{2}$ for all $k\leq n$ and the $H^{2}$ bound only depends on $\phi^{0}, \Omega$ and $T$. Thus, we have $\|g''(\phi^{*n+1})\|_{L^{\infty}}, \|g'''(\phi^{*n+1})\|_{L^{\infty}}\leq C$, where the positive constant $C$ only depends on $g(\cdot), \phi^{0}, \Omega$ and $T$.
Therefore, for $n\geq1$, making the inner product of \eqref{BDF_2_1} and \eqref{BDF_2_2} with $\phi^{n+1}-\phi^{n}$, and using the similar argument for $n=0$, we obtain
\bry
&\dps\Big(\frac{\gamma^{\frac{3}{2}}_{n+2}}{1+\gamma_{n+2}}+G(\gamma_{*},\gamma_{*})\Big)\frac{\|\phi^{n+1}-\phi^{n}\|^{2}}{2\Dt_{n+1}}-\frac{\gamma^{\frac{3}{2}}_{n+1}}{1+\gamma_{n+1}}\frac{\|\phi^{n}-\phi^{n-1}\|^{2}}{2\Dt_{n}}\\
&\dps+\frac{\varepsilon^{2}}{2}\big[\|\Delta\phi^{n+1}\|^{2}-\|\Delta\phi^{n}\|^{2}\big]+\frac{\lambda}{2}\big[\|\nabla\phi^{n+1}\|^{2}-\|\nabla\phi^{n}\|^{2}\big]\\
&\dps\leq\frac{R^{n+1}}{\sqrt{E_{1}^{n}}}\Big(\Delta g'(\phi^{*,n+1}),\phi^{n+1}-\phi^{n}\Big)\\
&\dps\leq C(M,\gamma_{*})\Dt_{n+1}\|\Delta g'(\phi^{*,n+1})\|^{2}+G(\gamma_{*},\gamma_{*})\frac{\|\phi^{n+1}-\phi^{n}\|^{2}}{2\Dt_{n+1}}\\
&\dps\leq C(M,\gamma_{*})\Dt_{n+1}+G(\gamma_{*},\gamma_{*})\frac{\|\nabla(\phi^{n+1}-\phi^{n})\|^{2}}{2\Dt_{n+1}}.
\ery
Summing up the above inequlity from 1 to $n$, and combining with \eqref{eq_p} gives
\bry
&\dps\frac{\gamma^{\frac{3}{2}}_{n+2}}{2(1+\gamma_{n+2})}\frac{\|\phi^{n+1}-\phi^{n}\|^{2}}{\Dt_{n+1}}
+\frac{\varepsilon^{2}}{2}\|\Delta\phi^{n+1}\|^{2}+\frac{\lambda}{2}\|\nabla\phi^{n+1}\|^{2}\\
&\dps \leq\frac{\gamma^{\frac{3}{2}}_{2}}{2(1+\gamma_{2})}\frac{\|\phi^{1}-\phi^{0}\|^{2}}{\Dt_{1}}
+\frac{\varepsilon^{2}}{2}\|\Delta\phi^{1}\|^{2}+\frac{\lambda}{2}\|\nabla\phi^{1}\|^{2}
+C(M,\gamma_{*})t_{n+1}\\[9pt]
&\dps \leq
\max\Big\{\frac{\gamma_{*}^{\frac{3}{2}}}{1+\gamma_{*}},1\Big\}\cdot\frac{\varepsilon^{2}}{2}\|\Delta(\phi^{0})\|^{2}
+C(M,\gamma_{*})(1+T).
\ery
Thus, we obtain the uniform $H^{2}$ bound \eqref{eq_p1} for the numerical solution $\phi^{n+1}$. Then the proof is completed.
\end{proof}
\subsection{$L^2$ gradient flow}
For the $L^{2}$ gradient flow, we will only state the $H^{2}$ bound of the numerical solution of VBDF2 scheme \eqref{BDF_2}. The proof is essentially same as the $H^{-1}$ gradient flow.
\begin{theorem}\label{Th2}
Assume $\phi^{0}\in H^{2}$, $0<\gamma_{n+1}\leq\gamma_{*}$ and $g(\cdot)\in C^{2}(\mathbb{R})$, then the solution of the VBDF2 scheme with $G_H=-I$ satisfies
\brr\label{eq_p2}
&\dps\frac{\gamma^{\frac{3}{2}}_{n+2}}{1+\gamma_{n+2}}\frac{\|\nabla(\phi^{n+1}-\phi^{n})\|^{2}}{2\Dt_{n+1}}
+\frac{\varepsilon^{2}}{2}\|\Delta\phi^{n+1}\|^{2}+\frac{\lambda}{2}\|\nabla\phi^{n+1}\|^{2}\\
\leq&\dps C(M,\gamma_{*})(1+T)+\max\Big\{\frac{\gamma_{*}^{\frac{3}{2}}}{1+\gamma_{*}},1\Big\}\cdot\frac{\varepsilon^{2}}{2}\|\Delta\phi^{0}\|^{2}.
\err
\end{theorem}
\begin{remark}
In particular, we have improved the $H^{2}$ bound results given in \cite{SX18}, where the proof of $H^{2}$ bound of the numerical solutions for $H^{-1}$ and $L^{2}$ gradient flows requires $\phi^{0}\in H^{4}$ and $H^{3}$ respectively. Here we only require $\phi^{0}\in H^{2}$ for both $L^{2}$ and $H^{-1}$ gradient flows, which is consistent with the requirement of the continuous problem to guarantee the $H^{2}$ bound of the solution.
\end{remark}
\section{Error analysis}
\label{sec:sect4}
In this section, we derive the error estimates of the VBDF2 scheme \eqref{BDF_2} for gradient flows.
Denote $e^{n}=\phi^{n}-\phi(t_{n})$ and $s^{n}=r^{n}-r(t_{n})$ with $e^{0}=s^{0}=0,$ and $\phi(t_{*,n+1})=\phi(t_{n})+\gamma_{n+1}(\phi(t_{n})-\phi(t_{n-1})).$ For simplicity, we only consider the case $V(\xi)=2-\xi$ hereafter.
\subsection{$H^{-1}$ gradient flow}
\begin{theorem}\label{th2}
For the $H^{-1}$ gradient flow, assume that $\phi^{0}:=\phi(\x,0)\in H^{2}, 0<\gamma_{n}\leq\gamma_{*},$ $\Dt_{1}\leq  \Dt^{\frac{4}{3}}$, and $g(\cdot)\in C^{3}(\mathbb{R})$. In addition, we suppose that
$$\phi_{t}\in L^{\infty}(0,T;L^{2})\cap L^{2}(0,T;L^{4}),\phi_{tt}\in L^{2}(0,T;H^{1}), \phi_{ttt}\in L^{2}(0,T;H^{-1}).$$
Then it holds that
\brr\label{error1}
&\dps\frac{\gamma^{\frac{3}{2}}_{n+2}}{2(1+\gamma_{n+2})}\frac{\|\nabla^{-1}(e^{n+1}-e^{n})\|^{2}}{\Dt_{n+1}}
+\frac{\varepsilon^{2}}{2}\|\nabla e^{n+1}\|^{2}+\frac{\lambda}{2}\|e^{n+1}\|^{2}+|s^{n+1}|^{2}\\[7pt]
\leq &\dps C\exp(T)\Big[\Dt^{2}\int_{0}^{T}\big(\|\phi_{t}(s)\|^{2}_{L^{4}}+\|\phi_{tt}(s)\|^{2}\big)ds\\
&\dps+\Dt^{4}\int_{0}^{T}\big(\|\phi_{tt}(s)\|^{2}_{H^{1}}+\|\phi_{ttt}(s)\|^{2}_{H^{-1}}\big)ds\Big],
\err
where the positive constant $C$ only depends $\phi^{0}, g(\cdot), \Omega, T$ and $\gamma_{*}.$
Moreover, we have
\brr\label{error2}
&\dps\frac{\gamma^{\frac{3}{2}}_{n+2}}{2(1+\gamma_{n+2})}\frac{\|\nabla^{-1}(e^{n+1}-e^{n})\|^{2}}{\Dt_{n+1}}
+\frac{\varepsilon^{2}}{2}\|\nabla e^{n+1}\|^{2}+\frac{\lambda}{2}\|e^{n+1}\|^{2}\\[7pt]
\leq &\dps C\exp(T)
\Dt^{4}\int_{0}^{T}\big(\|\phi_{t}(s)\|^{2}_{L^{4}}+\|\phi_{tt}(s)\|^{2}_{H^{1}}+\|\phi_{ttt}(s)\|^{2}_{H^{-1}}\big)ds.
\err
\end{theorem}
\begin{proof}
It follows from Lemma \ref{Th1} and \ref{Th3} that $\|\phi(t)\|_{H^{2}}, \|\phi^{n}\|\leq C,$ in which the positive constant $C$ depends on $\phi^{0}, \Omega, T$ and $\gamma_{*}$. By the Sobolev embedding theorem $H^{2}\subseteq L^{\infty}$, we have
\bq\label{equ1}
|g(\phi)|,|g'(\phi)|, |g''(\phi)|, |g(\phi^{n})|, |g'(\phi^{n})|, |g''(\phi^{n})|\leq C
\eq
 for all $n=0,1,\cdots.$ We
use the expression $\mathcal{A}\lesssim\mathcal{B}$ to mean that $\mathcal{A}\leq C\mathcal{B}$, hereafter. Direct calculation gives
 \beq
 \dps r_{tt}=\frac{-1}{4\sqrt{(E_{1}(\phi))^{3}}}\Big(\int_{\Omega}g'(u)\phi_{t}d\x\Big)^{2}+\frac{1}{2\sqrt{E_{1}(\phi)}}\int_{\Omega}\big[g''(\phi)\phi^{2}_{t}+g'(\phi)\phi_{tt}\big]d\x.
 \eeq
Combining with \eqref{equ1}, we deduce that
\bq\label{equa1}
\int_{0}^{t_{n+1}}|r_{tt}|^{2}dt\lesssim\int_{0}^{t_{n+1}}(\|\phi_{t}\|^{2}_{L^{4}}+\|\phi_{tt}\|^{2})dt.
\eq
For $n\geq1$, the error equations can be derived from  \eqref{re_prob} and \eqref{BDF_2}, as follows
 \begin{subequations}\label{err}
 \begin{align}
 \dps F^{n+1}_{2}e+\varepsilon^{2}\Delta^{2} e^{n+1}-\lambda\Delta e^{n+1}=&\dps\frac{s^{n+1}V(\xi^{n+1})}{\sqrt{E_{1}^{n}}}\Delta g'(\phi^{*n+1})+\dps J^{n}_{1}
 +J^{n}_{2}+T^{n}_{1}+T^{n}_{2},\label{err_1}\\
 \dps s^{n+1}-s^{n}=&\dps\frac{V(\xi^{n+1})}{2\sqrt{E_{1}^{n}}}\int_{\Omega}g'(\phi^{*n+1})(e^{n+1}-e^{n})d\x \label{err_2}\\
 &\dps+\frac{1}{2}\int_{\Omega}J^{n}_{3}\cdot
(\phi(t_{n+1})-\phi(t_{n}))d\x-v^{n}_{1}+v^{n}_{2}, \notag
 \end{align}
 \end{subequations}
where
\bry
J^{n}_{1}:=& \dps r(t_{n+1})V(\xi^{n+1})\Big[\frac{\Delta g'(\phi^{*,n+1})}{\sqrt{E_{1}^{n}}}-\frac{\Delta g'(\phi(t_{*,n+1}))}{\sqrt{E_{1}(\phi(t_{n+1}))}}\Big],\\[8pt] J^{n}_{2}:=&\dps\frac{r(t_{n+1})(V(\xi^{n+1})-1)}{\sqrt{E_{1}(\phi(t_{n+1}))}}\Delta g'(\phi(t_{*,n+1})),\\[8pt]
J^{n}_{3}:=&\dps \frac{V(\xi^{n+1})g'(\phi^{*,n+1})}{\sqrt{E_{1}^{n}}}-\frac{g'(\phi(t_{n+1}))}{\sqrt{E_{1}(\phi(t_{n+1}))}},
\ery
and the truncation errors are given by
 \begin{subequations}\label{err1}
 \begin{align}
 T^{n}_{1}&\dps=\phi_{t}(t_{n+1})-\partial_{t}(\Pi_{2,n}\phi(t))|_{t=t_{n+1}}, ~T^{n}_{2}\dps=\Delta g'(\phi(t_{*,n+1}))-\Delta g'(\phi(t_{n+1})),\label{tr_err1}\\
v^{n}_{1}&\dps=r(t_{n+1})-r(t_{n})-\Dt_{n+1}r_{t}(t_{n+1})=\int_{t_{n}}^{t_{n+1}}(t_{n}-s)r_{tt}(s)ds,\label{tr_err3}\\
v^{n}_{2}&\dps=\Big(\frac{g'(\phi(t_{n+1}))}{2\sqrt{E_{1}(\phi(t_{n+1}))}},\int_{t_{n}}^{t_{n+1}}(t_{n}-s)\phi_{tt}(s)ds\Big).\label{tr_err4}
 \end{align}
 \end{subequations}
Taking the $L^{2}$ inner product \eqref{err_1} and \eqref{err_2} with $(-\Delta)^{-1}(e^{n+1}-e^{n})$ and $2s^{n+1}$, respectively, and summing them up, gives
\bry
&\dps\frac{1}{2}\Big[\Big(\frac{\gamma^{\frac{3}{2}}_{n+2}}{1+\gamma_{n+2}}+G(\gamma_{*},\gamma_{*})\Big)\frac{\|\nabla^{-1}(e^{n+1}-e^{n})\|^{2}}{\Dt_{n+1}}-\frac{\gamma^{\frac{3}{2}}_{n+1}}{1+\gamma_{n+1}}\frac{\|\nabla^{-1}(e^{n}-e^{n-1})\|^{2}}{\Dt_{n}}\Big]\\[7pt]
&\dps+\frac{\varepsilon^{2}}{2}[\|\nabla e^{n+1}\|^{2}-\|\nabla e^{n}\|^{2}]+\frac{\lambda}{2}[\|e^{n+1}\|^{2}-\|e^{n}\|^{2}]+|s^{n+1}|^{2}-|s^{n}|^{2}\\[7pt]
\leq&\dps\big(\nabla^{-1}(J^{n}_{1}+J^{n}_{2}+T^{n}_{1}+T^{n}_{2}),\nabla^{-1}(e^{n+1}-e^{n})\big)+s^{n+1}\int_{\Omega}J^{n}_{3}\cdot
(\phi(t_{n+1})-\phi(t_{n}))d\x\\[7pt]
&\dps+2s^{n+1}(-v^{n}_{1}+v^{n}_{2})\\[7pt]
\leq&\dps C(\gamma_{*})\Dt_{n+1}\big[\|\nabla^{-1}J^{n}_{1}\|^{2}+\|\nabla^{-1}J^{n}_{2}\|^{2}+\|\nabla^{-1}T^{n}_{1}\|^{2}+\|\nabla^{-1}T^{n}_{2}\|^{2}\big]\\[7pt]
&\dps+G(\gamma_{*},\gamma_{*})\frac{\|\nabla^{-1}(e^{n+1}-e^{n})\|^{2}}{2\Dt_{n+1}}+C\Dt_{n+1}\|\phi_{t}\|_{L^{\infty}(0,T;L^{2})}\big[|s^{n+1}|^{2}+\|J^{n}_{3}\|^{2}\big]\\
&\dps+\Dt_{n+1}|s^{n+1}|^{2}+\frac{2}{\Dt_{n+1}}(|v^{n}_{1}|^{2}+|v^{n}_{2}|^{2}),
\ery
in which we have used the inequality \eqref{ideq1} and the identities \eqref{Id}.
Thus we can deduce from the above inequality that
\brr\label{equ2}
&\dps\frac{1}{2}\Big[\frac{\gamma^{\frac{3}{2}}_{n+2}}{1+\gamma_{n+2}}\frac{\|\nabla^{-1}(e^{n+1}-e^{n})\|^{2}}{\Dt_{n+1}}-\frac{\gamma^{\frac{3}{2}}_{n+1}}{1+\gamma_{n+1}}\frac{\|\nabla^{-1}(e^{n}-e^{n-1})\|^{2}}{\Dt_{n}}\Big]\\[7pt]
&\dps+\frac{\varepsilon^{2}}{2}\Big[\|\nabla(e^{n+1})\|^{2}-\|\nabla(e^{n})\|^{2}\Big]+\frac{\lambda}{2}\Big[\|e^{n+1}\|^{2}-\|e^{n}\|^{2}\Big]+|s^{n+1}|^{2}-|s^{n}|^{2}\\[7pt]
\lesssim&\dps \Dt_{n+1}\Big[\|\nabla^{-1}J^{n}_{1}\|^{2}+\|\nabla^{-1}J^{n}_{2}\|^{2}+\|J^{n}_{3}\|^{2}+|s^{n+1}|^{2}+\|\nabla^{-1}T^{n}_{1}\|^{2}\\[7pt]
&\dps+\|\nabla^{-1}T^{n}_{2}\|^{2}\Big]+\frac{2(|v^{n}_{1}|^{2}+|v^{n}_{2}|^{2})}{\Dt_{n+1}}.
\err
Note that
\bq\label{equat1}
\dps \nabla^{-1}J^{n}_{1}
=-r(t_{n+1})V(\xi^{n+1})\Big[\frac{\nabla I^{n}_{1}}{\sqrt{E_{1}^{n}}}+\nabla g'(\phi(t_{*,n+1}))\cdot I^{n}_{2}\Big],
\eq
where
\bq\label{equ6}
I^{n}_{1}:=g'(\phi^{*,n+1})-g'(\phi(t_{*,n+1})),\ \ I^{n}_{2}:=\frac{1}{\sqrt{E_{1}^{n}}}-\frac{1}{\sqrt{E_{1}(\phi(t_{n+1}))}}.
\eq
Therefore, it follows from \eqref{equat1} and \eqref{equ6} that
\bq\label{equ3}
\|\nabla^{-1}J^{n}_{1}\|^{2}\lesssim \|\nabla I^{n}_{1}\|^{2}+|I^{n}_{2}|^{2}.
\eq
From the definition of $J^{n}_{2}$ and $V(\xi)=2-\xi$, we have
\brr\label{equa5}
\|\nabla^{-1}J^{n}_{2}\|^{2}&\lesssim\dps |1-V(\xi^{n+1})|^{2}\dps\lesssim \Big|\xi^{n+1}-1\Big|^{2}\\
&\dps\lesssim\Big|\frac{r^{n+1}}{\sqrt{E_{1}^{n}}}-\frac{r(t_{n+1})}{\sqrt{E_{1}(\phi(t_{n+1}))}}\Big|^{2}\\[7pt]
&\dps\lesssim\Big|\frac{s^{n+1}}{\sqrt{E_{1}^{n}}}+r(t_{n+1})\cdot I^{n}_{2}\Big|^{2}\\[9pt]
&\dps\lesssim |s^{n+1}|^{2}+|I^{n}_{2}|^{2}.
\err
Then, it holds that
\brr\label{equ4}
\|J^{n}_{3}\|^{2}=&\dps\Big\|\frac{(V(\xi^{n+1})-1)g'(\phi^{*,n+1})}{\sqrt{E_{1}^{n}}}+\frac{I^{n}_{1}}{\sqrt{E_{1}^{n}}}+\frac{\Delta^{-1}T^{n}_{2}}{\sqrt{E_{1}^{n}}}+g'(\phi(t_{n+1}))\cdot I^{n}_{2}\Big\|^{2}\\[11pt]
\lesssim &\dps |1-V(\xi^{n+1})|^{2}+\|I^{n}_{1}\|^{2}+\|\Delta^{-1}T^{n}_{2}\|^{2}+|I^{n}_{2}|^{2}\\[7pt]
\lesssim&\dps |s^{n+1}|^{2}+\|I^{n}_{1}\|^{2}+\|\Delta^{-1}T^{n}_{2}\|^{2}+|I^{n}_{2}|^{2}.
\err
From \eqref{equ2}-\eqref{equ4}, it follows that
\brr\label{equ5}
&\dps\frac{1}{2}\Big[\frac{\gamma^{\frac{3}{2}}_{n+2}}{1+\gamma_{n+2}}\frac{\|\nabla^{-1}(e^{n+1}-e^{n})\|^{2}}{\Dt_{n+1}}-\frac{\gamma^{\frac{3}{2}}_{n+1}}{1+\gamma_{n+1}}\frac{\|\nabla^{-1}(e^{n}-e^{n-1})\|^{2}}{\Dt_{n}}\Big]\\[7pt]
&\dps+\frac{\varepsilon^{2}}{2}\big[\|\nabla e^{n+1}\|^{2}-\|\nabla e^{n}\|^{2}\big]
+\frac{\lambda}{2}\big[\|e^{n+1}\|^{2}-\|e^{n}\|^{2}\big]+|s^{n+1}|^{2}-|s^{n}|^{2}\\[7pt]
\lesssim&\dps \Dt_{n+1}\big[|s^{n+1}|^{2}+\|I^{n}_{1}\|^{2}+\|\nabla I^{n}_{1}\|^{2}+|I^{n}_{2}|^{2}
+\|\nabla^{-1}T^{n}_{1}\|^{2}+\|\nabla^{-1}T^{n}_{2}\|^{2}\\
&\dps+\|\Delta^{-1}T^{n}_{2}\|^{2}\big]
+\frac{2}{\Dt_{n+1}}(|v^{n}_{1}|^{2}+|v^{n}_{2}|^{2}).
\err
By the definition of $I^{n}_{1}$ and $I^{n}_{2}$ in \eqref{equ6}, we can get the following estimates
\brr\label{equ7}
\|I^{n}_{1}\|^{2}=&\dps \|g'(\phi^{*,n+1})-g'(\phi(t_{*,n+1}))\|^{2}\lesssim \|e^{n}\|^{2}+\|e^{n-1}\|^{2},\\[7pt]
|I^{n}_{2}|^{2}=&\dps\Big|\frac{E_{1}(\phi(t_{n+1}))-E^{n}_{1}}{\sqrt{E_{1}^{n}}\sqrt{E_{1}(\phi(t_{n+1}))}\big(\sqrt{E_{1}(\phi(t_{n+1}))}+\sqrt{E_{1}^{n}}\big)}\Big|^{2}\\[13pt]
\lesssim&\dps |E_{1}(\phi(t_{n+1}))-E^{n}_{1}|^{2}\\[7pt]
\lesssim&\dps |E_{1}(\phi(t_{n+1}))-E_{1}(\phi(t_{n}))|^{2}+|E_{1}(\phi(t_{n}))-E^{n}_{1}|^{2}\\[7pt]
\lesssim&\dps\Dt_{n+1}\int_{t_{n}}^{t_{n+1}}\|\phi_{t}(s)\|^{2}ds+\|e^{n}\|^{2}.
\err
Moreover, we have following estimate for $\nabla I^{n}_{1}$ by using the H\"{o}lder's inequality and the Sobolev embedding theorem, $H^{1}\subset L^{6}$,
\brr\label{est_I}
\|\nabla I^{n}_{1}\|^{2}=&\dps \|g''(\phi^{*,n+1})\nabla(\phi^{*,n+1}-\phi(t_{*,n+1}))\\[7pt]
&+(g''(\phi^{*,n+1})-g''(\phi(t_{*,n+1})))\nabla\phi(t_{*,n+1})\|^{2}\\[7pt]
\lesssim&\dps \|\nabla e^{n}\|^{2}+\|\nabla e^{n-1}\|^{2}+\|(\phi^{*,n+1}-\phi(t_{*,n+1}))\nabla\phi(t_{*,n+1})\|^{2}\\[7pt]
\lesssim&\dps \|\nabla e^{n}\|^{2}+\|\nabla e^{n-1}\|^{2}+\|\phi^{*,n+1}-\phi(t_{*,n+1})\|^{2}_{L^{6}}\|\nabla\phi(t_{*,n+1})\|^{2}_{L^{3}}\\[7pt]
\lesssim&\dps \|\nabla e^{n}\|^{2}+\|\nabla e^{n-1}\|^{2}+\|\phi^{*,n+1}-\phi(t_{*,n+1})\|^{2}_{H^{1}}\|\phi(t_{*,n+1})\|^{2}_{H^{2}}\\[7pt]
\lesssim&\dps \|\nabla e^{n}\|^{2}+\|\nabla e^{n-1}\|^{2}+\|e^{n}\|^{2}+\|e^{n-1}\|^{2}\\[7pt]
\err
For the truncation errors, we have the following estimates, seeing also \cite{CWYZ19}
\brr\label{equ8}
\|\nabla^{-1}T^{n}_{1}\|^{2}\lesssim&\dps (\Dt_{n}+\Dt_{n+1})^{3}\int_{t_{n-1}}^{t_{n+1}}\|\nabla^{-1}\phi_{ttt}(s)\|^{2}ds,\\
\|\Delta^{-1}T^{n}_{2}\|^{2}\lesssim&\dps \|\phi(t_{n+1})-\phi(t_{*,n+1})\|^{2}\lesssim (\Dt_{n}+\Dt_{n+1})^{3}\int_{t_{n-1}}^{t_{n+1}}\|\phi_{tt}(s)\|^{2}ds,\\
\|\nabla^{-1}T^{n}_{2}\|^{2}\lesssim&\dps \|\phi(t_{n+1})-\phi(t_{*,n+1})\|^{2}+\|\nabla(\phi(t_{n+1})-\phi(t_{*,n+1}))\|^{2}\\
\lesssim&\dps (\Dt_{n}+\Dt_{n+1})^{3}\int_{t_{n-1}}^{t_{n+1}}\|\phi_{tt}(s)\|^{2}_{H^{1}}ds,\\
|v^{n}_{1}|^{2}\lesssim&\dps \Dt^{3}_{n+1}\int_{t_{n}}^{t_{n+1}}|r_{tt}(s)|^{2}ds,\\
|v^{n}_{2}|^{2}\lesssim&\dps \Dt_{n+1}^{3}\Big\|\frac{g'(\phi(t_{n+1}))}{2\sqrt{E_{1}(\phi(t_{n+1}))}}\Big\|^{2}\int_{t_{n}}^{t_{n+1}}\|\phi_{tt}(s)\|^{2}ds\\
\lesssim&\dps \Dt^{3}_{n+1}\int_{t_{n}}^{t_{n+1}}\|\phi_{tt}(s)\|^{2}ds.
\err
Taking the inequalities \eqref{equ7}--\eqref{equ8} into \eqref{equ5} gives
\bry
&\dps\frac{1}{2}\Big[\frac{\gamma^{\frac{3}{2}}_{n+2}}{1+\gamma_{n+2}}\frac{\|e^{n+1}-e^{n}\|^{2}}{\Dt_{n+1}}-\frac{\gamma^{\frac{3}{2}}_{n+1}}{1+\gamma_{n+1}}\frac{\|e^{n}-e^{n-1}\|^{2}}{\Dt_{n}}\Big]\\[8pt]
&\dps+\frac{\varepsilon^{2}}{2}[\|\nabla(e^{n+1})\|^{2}-\|\nabla(e^{n})\|^{2}]+\frac{\lambda}{2}[\|e^{n+1}\|^{2}-\|e^{n}\|^{2}]+|s^{n+1}|^{2}-|s^{n}|^{2}\\
\ery
\bry
\lesssim&\dps \Dt_{n+1}\big[|s^{n+1}|^{2}+\|e^{n}\|^{2}+\|e^{n-1}\|^{2}\big]+\Dt^{2}_{n+1}\int_{t_{n}}^{t_{n+1}}\big(\|\phi_{t}(s)\|^{2}+\|\phi_{tt}(s)\|^{2}\\[8pt]
&\dps+|r_{tt}(s)|^{2}\big)ds+\Dt_{n+1}(\Dt_{n}+\Dt_{n+1})^{3}\int_{t_{n-1}}^{t_{n+1}}\big(\|\phi_{tt}(s)\|^{2}+\|\phi_{ttt}(s)\|^{2}\big)ds\\[8pt]
\lesssim&\dps \Dt_{n+1}\big[|s^{n+1}|^{2}+\|e^{n}\|^{2}+\|e^{n-1}\|^{2}\big]+\Dt^{2}\int_{t_{n}}^{t_{n+1}}\big(\|\phi_{t}(s)\|^{2}+\|\phi_{tt}(s)\|^{2}\\
&\dps+|r_{tt}(s)|^{2}\big)ds+\Dt^{4}\int_{t_{n-1}}^{t_{n+1}}\big(\|\phi_{tt}(s)\|^{2}_{H^{1}}+\|\phi_{ttt}(s)\|^{2}_{H^{-1}}\big)ds.
\ery
Then, we sum up the above inequality from 1 to $n$ to get
\brr\label{equ9}
&\dps\frac{\gamma^{\frac{3}{2}}_{n+2}}{2(1+\gamma_{n+2})}\frac{\|e^{n+1}-e^{n}\|^{2}}{\Dt_{n+1}}
+\frac{\varepsilon^{2}}{2}\|\nabla e^{n+1}\|^{2}+\frac{\lambda}{2}\|e^{n+1}\|^{2}+|s^{n+1}|^{2}\\[11pt]
&\dps-\Big[\frac{\gamma^{\frac{3}{2}}_{2}}{2(1+\gamma_{2})}\frac{\|e^{1}\|^{2}}{\Dt_{1}}+\frac{\varepsilon^{2}}{2}\|\nabla e^{1}\|^{2} +\frac{\lambda}{2}\|e^{1}\|^{2}+|s^{1}|^{2}\Big]\\[11pt]
\lesssim&\dps\sum_{k=1}^{n}\Dt_{k+1}\Big[|s^{k+1}|^{2}+\|e^{k}\|^{2}+\|e^{k-1}\|^{2}\Big]
+\Dt^{2}\int_{t_{1}}^{t_{n+1}}\Big(\|\phi_{t}(s)\|^{2}+\|\phi_{tt}(s)\|^{2}\\
&\dps+|r_{tt}(s)|^{2}\Big)ds+\Dt^{4}\int_{0}^{t_{n+1}}\big(\|\phi_{tt}(s)\|^{2}_{H^{1}}+\|\phi_{ttt}(s)\|^{2}_{H^{-1}}\big)ds.\\
\err
For the first time step, that is n=0, we use the first order scheme \eqref{eq1}, and the error equations are given as follows
\bry
 &\dps\frac{e^{1}}{\Dt_{1}}+\varepsilon^{2}\Delta^{2}e^{1}+\lambda\Delta e^{1}=\dps\frac{s^{1}}{\sqrt{E^{0}+C_{0}}}\Delta g'(\phi^{0})+\frac{r(t_{1})\Delta g'(\phi^{0})}{\sqrt{E^{0}+C_{0}}}-\frac{r(t_{1})\Delta g'(\phi^{0})}{\sqrt{E_{1}(\phi(t_{1}))+C_{0}}}\\
  &\hspace{3.7cm}\dps+\Delta (g(\phi^{0})-g(\phi(t_{1})))-\Dt_{1}^{-1}\int_{0}^{t_{1}}s\phi_{tt}(s)ds,\\
 &\dps s^{1}=\frac{\int_{\Omega}g'(\phi^{0})e^{1}d\x}{2\sqrt{E_{0}^{n}+C_{0}}}
 +\frac{1}{2}\int_{\Omega}\Big[\frac{g(\phi^{0})}{\sqrt{E^{0}_{1}+C_{0}}}-\frac{g(\phi(t_{1}))}{\sqrt{E_{1}(\phi(t_{1}))}}\Big]
(\phi(t_{1})-\phi(t_{0}))d\x\\[7pt]
&\hspace{0.7cm}+v^{0}_{1}+v^{0}_{2}.
\ery
Taking a similar argument of the case $n\geq1$, we can deduce the following estimate with the assumption $\Dt_{1}\leq \Dt^{4/3}$
\bq\label{equa2}
 \dps\frac{\|\nabla^{-1}e^{1}\|^{2}}{2\Dt_{1}}+\frac{\varepsilon^{2}}{2}\|\nabla e^{1}\|^{2}+\frac{\lambda}{2} \|e^{1}\|^{2}+\|s^{1}\|^{2}\lesssim \Dt_{1}^{3}\lesssim\Dt^{4},
\eq
where the positive constant $C$ depends on the $\|\phi_{t}\|_{H^{1}},\|\phi_{tt}\|_{H^{-1}}$ and $|r_{tt}|$. Combining with \eqref{equa1}, \eqref{equ9} and \eqref{equa2}, and together with the discrete Gronwall's lemma, we obtain the desired estimate \eqref{equa3}.

Next, we prove the estimate \eqref{equa4}, which indicates the second order accuracy of the numerical solution $\phi^{n}$ in time. For $n\geq1$, the error equation can be written as
\beq
F^{n+1}_{2}e+\varepsilon^{2}\Delta^{2} e^{n+1}-\lambda\Delta e^{n+1}=\big(\xi^{n+1}V(\xi^{n+1})-1\big)\Delta g'(\phi^{*,n+1})+\Delta I^{n}_{1}+T^{n}_{1}+T^{n}_{2}.
\eeq
Taking the $L^{2}$ inner product of the above equation with $(-\Delta)^{-1}(e^{n+1}-e^{n})$, and employing the Young's inequality, we get
\bry
&\dps\frac{1}{2}\Big[\big(\frac{\gamma^{\frac{3}{2}}_{n+2}}{1+\gamma_{n+2}}+G(\gamma_{*},\gamma_{*})\big)\frac{\|\nabla^{-1}(e^{n+1}-e^{n})\|^{2}}{\Dt_{n+1}}-\frac{\gamma^{\frac{3}{2}}_{n+1}}{1+\gamma_{n+1}}\frac{\|\nabla^{-1}(e^{n}-e^{n-1})\|^{2}}{\Dt_{n}}\Big]\\[7pt]
&\dps+\frac{\varepsilon^{2}}{2}[\|\nabla e^{n+1}\|^{2}-\|\nabla e^{n}\|^{2}]+\frac{\lambda}{2}[\|e^{n+1}\|^{2}-\|e^{n}\|^{2}]\\[7pt]
\leq&\dps C(\gamma_{*})\Dt_{n+1}\Big[|1-\xi^{n+1}V(\xi^{n+1})|^{2}+\|\nabla I^{n}_{1}\|^{2}+\|\nabla^{-1}T^{n}_{1}\|^{2}+\|\nabla^{-1}T^{n}_{2}\|^{2}\Big]\\
&\dps+G(\gamma_{*},\gamma_{*})\frac{\|e^{n+1}-e^{n}\|^{2}}{2\Dt_{n+1}}.
\ery
It follows from \eqref{equa5}, \eqref{equ7} and \eqref{equa3} that
\beq
|1-\xi^{n+1}V(\xi^{n+1})|^{2}=|1-\xi^{n+1}|^{4}\lesssim\Dt^{4}.
\eeq
Therefore, using the estimates of $\|\nabla I^{n}_{1}\|^{2}$, $\|\nabla^{-1}T^{n}_{1}\|^{2}$ and $\|\nabla^{-1}T^{n}_{2}\|^{2}$ in the above discussion, and the discrete Gronwall's lemma, one can easily derive the desired estimate \eqref{equa4}. Then we complete the proof.
\end{proof}

\subsection{$L^{2}$ gradient flow}
Since the proof of the associated error estimates for $L^{2}$ gradient flow is essentially similar to the case of $H^{-1}$ gradient flow, we only state it below and leave the proof for the interested readers.
\begin{theorem}\label{th1}
For the $L^{2}$ gradient flow, assume that $\phi^{0}:=\phi(\x,0)\in H^{2}, 0<\gamma_{n}\leq\gamma_{*},$ $\Dt_{1}\leq \Dt^{\frac{4}{3}}$, and $g(\cdot)\in C^{2}(\mathbb{R})$. Provided that
$$\phi_{t}\in L^{\infty}(0,T;L^{2})\cap L^{2}(0,T;L^{4}),\phi_{tt}, \phi_{ttt}\in L^{2}(0,T;L^{2}),$$
we have
\brr\label{equa3}
&\dps\frac{\gamma^{\frac{3}{2}}_{n+2}}{2(1+\gamma_{n+2})}\frac{\|e^{n+1}-e^{n}\|^{2}}{\Dt_{n+1}}
+\frac{\varepsilon^{2}}{2}\|\nabla e^{n+1}\|^{2}+\frac{\lambda}{2}\|e^{n+1}\|^{2}+|s^{n+1}|^{2}\\[11pt]
\leq &\dps C\exp(T)\Big[\Dt^{2}\int_{0}^{T}\big(\|\phi_{t}(s)\|^{2}_{L^{4}}+\|\phi_{tt}(s)\|^{2}\big)ds\\
&\dps+\Dt^{4}\int_{0}^{T}\big(\|\phi_{tt}(s)\|^{2}+\|\phi_{ttt}(s)\|^{2}\big)ds\Big].
\err
where the positive constant $C$ only depends $\phi^{0}, g(\cdot), \Omega, T$ and $\gamma_{*}.$
Moreover,
\brr\label{equa4}
&\dps\frac{\gamma^{\frac{3}{2}}_{n+2}}{2(1+\gamma_{n+2})}\frac{\|e^{n+1}-e^{n}\|^{2}}{\Dt_{n+1}}
+\frac{\varepsilon^{2}}{2}\|\nabla e^{n+1}\|^{2}+\frac{\lambda}{2}\|e^{n+1}\|^{2}\\[11pt]
\leq &\dps C\exp(T)
\Dt^{4}\int_{0}^{T}\big(\|\phi_{t}(s)\|^{2}_{L^{4}}+\|\phi_{tt}(s)\|^{2}+\|\phi_{ttt}(s)\|^{2}\big)ds.
\err
\end{theorem}
\begin{remark}
For the more general $V(\xi)$ satisfying \eqref{chv}, we use Taylor expansion of the functions $L_{1}(\xi)=1-V(\xi)$ and $L(\xi)=1-\xi V(\xi)$ at $\xi=1$, then there exist real number $\eta_{1}$ and $\eta_{2}$ between $\xi$ and $1$ such that
\bq
1-V(\xi)=V^{'}(\eta_{1})(1-\xi),
\ \ 1-\xi V(\xi)=L^{''}(\eta_{2})(1-\xi)^{2},
\eq
Using this property, one can also easily derive the estimates \eqref{error1}, \eqref{error2}, \eqref{equa3}, and \eqref{equa4} for general case of $V(\xi)$ by following the above discussion for the special case $V(\xi)=2-\xi$. Note that we only give the error analysis for semi-discrete scheme \eqref{BDF_2} in this section. For the fully discrete scheme, the associated error estimate can be easily derived by following a similar argument for the semi-discrete problem. One can also refer to the fully discrete error analysis of the SAV approaches with the finite difference method \cite{LSR19} and the finite element method \cite{CMS20} for the spatial discretization on the uniform temporal mesh.
\end{remark}

\section{Numerical results}
\label{sec:sect5}
In this section, we present some numerical examples to validate the derived theoretical results of the proposed schemes in terms of stability and accuracy. For simplicity, we set $\sigma=1$ and $\lambda=1$ throughout the following numerical tests.

\subsection{Test of the convergence order}

We first consider the following gradient flows problem, subject to the periodic boundary condition:
\bq\label{ex2}
\begin{cases}
\begin{array}{r@{}l}
&\dps\frac{\partial \phi}{\partial t}+G_H\big(\varepsilon^{2}\Delta \phi+\phi(1-\phi^{2})\big)=0,
\quad \mbox{in } (0,2\pi)^2\times(0,T],\\[9pt]
&\phi(\x,0)=\phi_{0}(\x), \quad \forall \x\in (0,2\pi)^2,
\end{array}
\end{cases}
\eq
where $\varepsilon^{2}=0.01,$ $G_H$ is defined in \eqref{ef}, and
\beq
\phi_{0}=
\begin{cases}
\begin{array}{r@{}l}
&\sin x\sin y,\quad \mbox{ for the $L^{2}$ gradient flow, }\\[9pt]
&0.1(\sin3x\sin2y+\sin5x\sin5y),\quad \mbox{ for the $H^{-1}$ gradient flow.}
\end{array}
\end{cases}
 \eeq

We use the Fourier spectral method with $128\times128$ Fourier modes for the spatial discretization.
This used Fourier mode number has been checked to be large enough such that the spacial discretization error can be negligible compared to the error from temporal discretization, seeing \cite{HAX19}.
Since there is no analytical solution available for the above gradient flows, we use
the numerical solution of SAV BDF2 scheme \cite{HAX19} with uniform small enough
time step size $\Delta t=1e-5$ as the reference solution.
To check the temporal accuracy,
we firstly employed the proposed VBDF2 scheme \eqref{BDF_2} to  both
$L^2$ and $H^{-1}$ gradient flows \eqref{ex2} on the uniform mesh with some different functions of $V(\xi)$.
In Figure \ref{fig1},
the $L^{\infty}$ error at $T=1$ is presented as functions of the time
step sizes in log-log scale.
It is shown that the VBDF2 scheme \eqref{BDF_2} for numerical approximation to $\phi$ and $\xi$ achieve the expected convergence rate of second order and first order for all tested functions $V(\xi)$, respectively.
 Here the used auxiliary function $V(\xi)$ in Figure \ref{fig1} is set to be $2-\xi,\exp(1-\xi)$ and $1+\sin(1-\xi)$ for different simulation.
 All these functions satisfy the assumption condition \eqref{ap_xi} to guarantee the second order approximation to the unknown function $\phi$.
 Next, we check the accuracy of the proposed scheme \eqref{BDF_2} on a nonuniform temporal mesh to verify our theoretical convergence results in Theorem \ref{th2} and Theorem \eqref{th1}.
 The nonuniform temporal mesh $\{\widehat{t}_{n}\}_{n=0}^{N}$ used here is the uniform mesh $\{t_{n}=n\tau\}_{n=0}^{N}$ with $40\%$ perturbation.
 We set $V(\xi)=2-\xi$ in the VBDF2 scheme \eqref{BDF_2} for this nonuniform mesh. As shown in Table \ref{tabel1}, the proposed VBDF2 scheme \eqref{BDF_2} is of second order accuracy for the phase function $\phi$ on the nonuniform meshes, even for the cases with the adjacent time step ratio $\gamma_{n}>4.8645$.
 Thus the $\gamma_{**}=4.8645$ may be not the optimal upper bound constrain for the adjacent time step ratios, which worth further investigation.

\begin{figure*}[htbp]
\begin{minipage}[t]{0.49\linewidth}
\centerline{\includegraphics[scale=0.45]{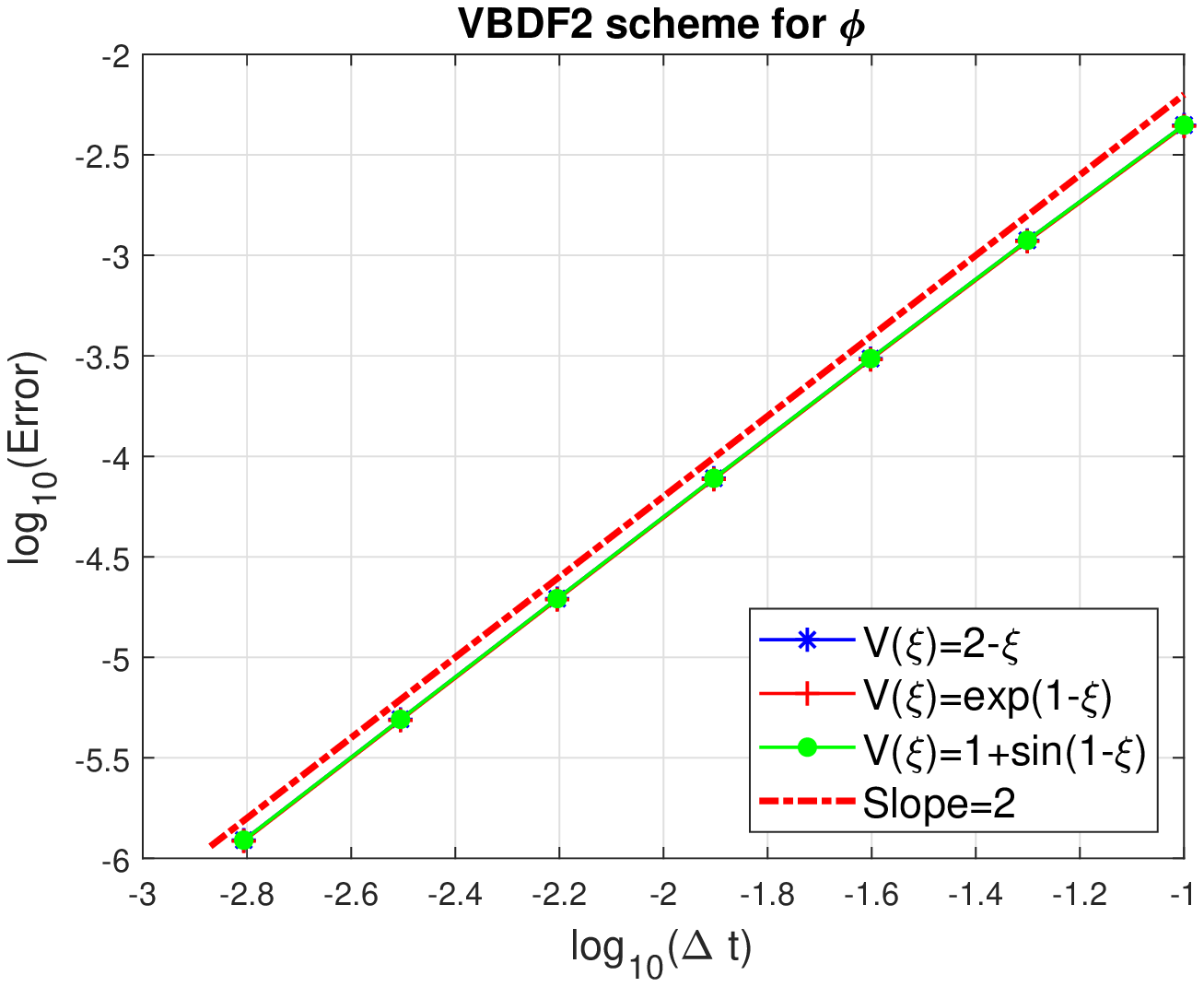}}
\centerline{\hspace{8cm}(a) $L^{2}$ gradient flow}
\end{minipage}
\begin{minipage}[t]{0.49\linewidth}
\centerline{\includegraphics[scale=0.45]{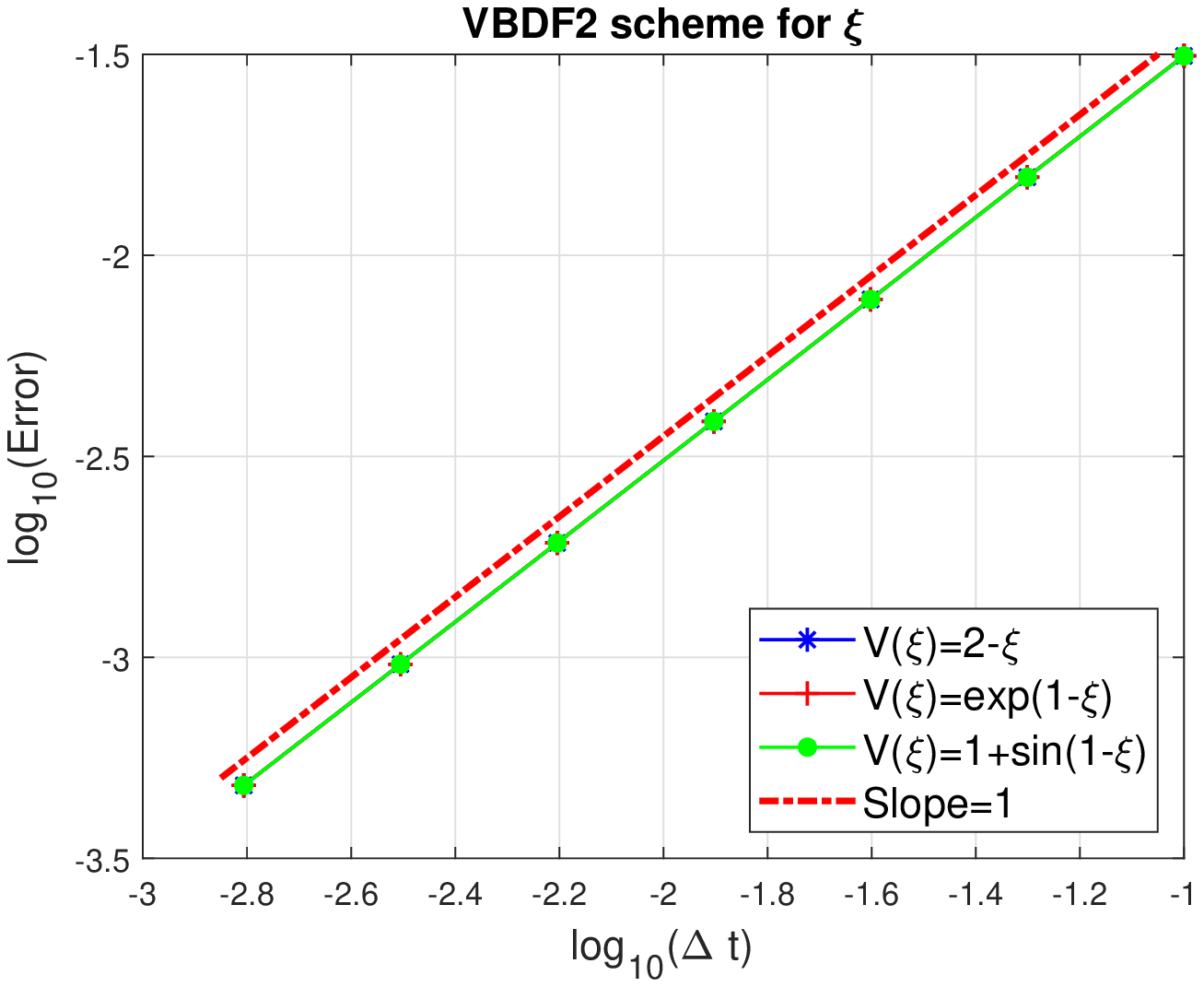}}
\centerline{}
\end{minipage}
\vskip 3mm
\begin{minipage}[t]{0.49\linewidth}
\centerline{\includegraphics[scale=0.45]{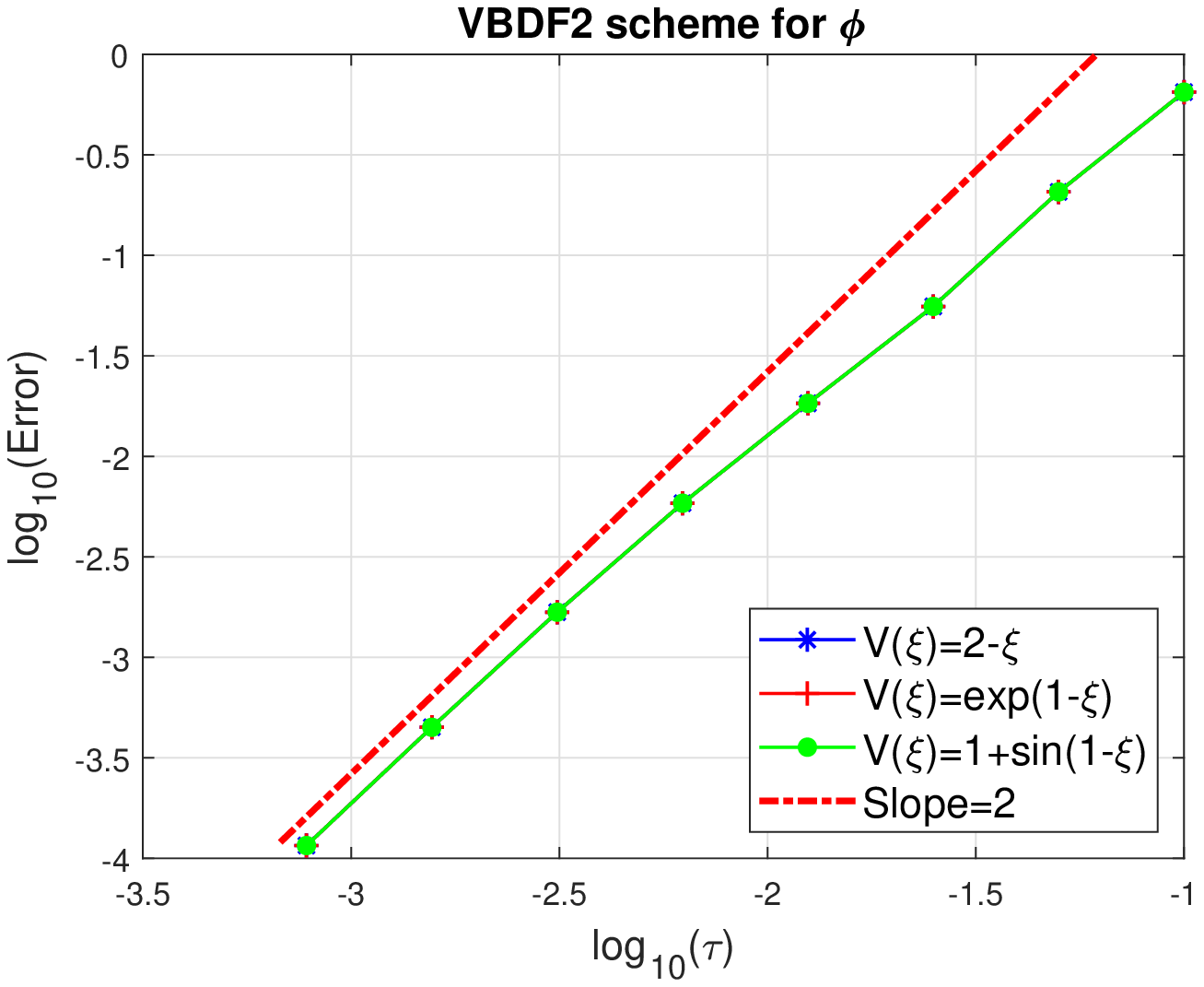}}
\centerline{\hspace{8cm}(b) $H^{-1}$ gradient flow}
\end{minipage}
\begin{minipage}[t]{0.49\linewidth}
\centerline{\includegraphics[scale=0.45]{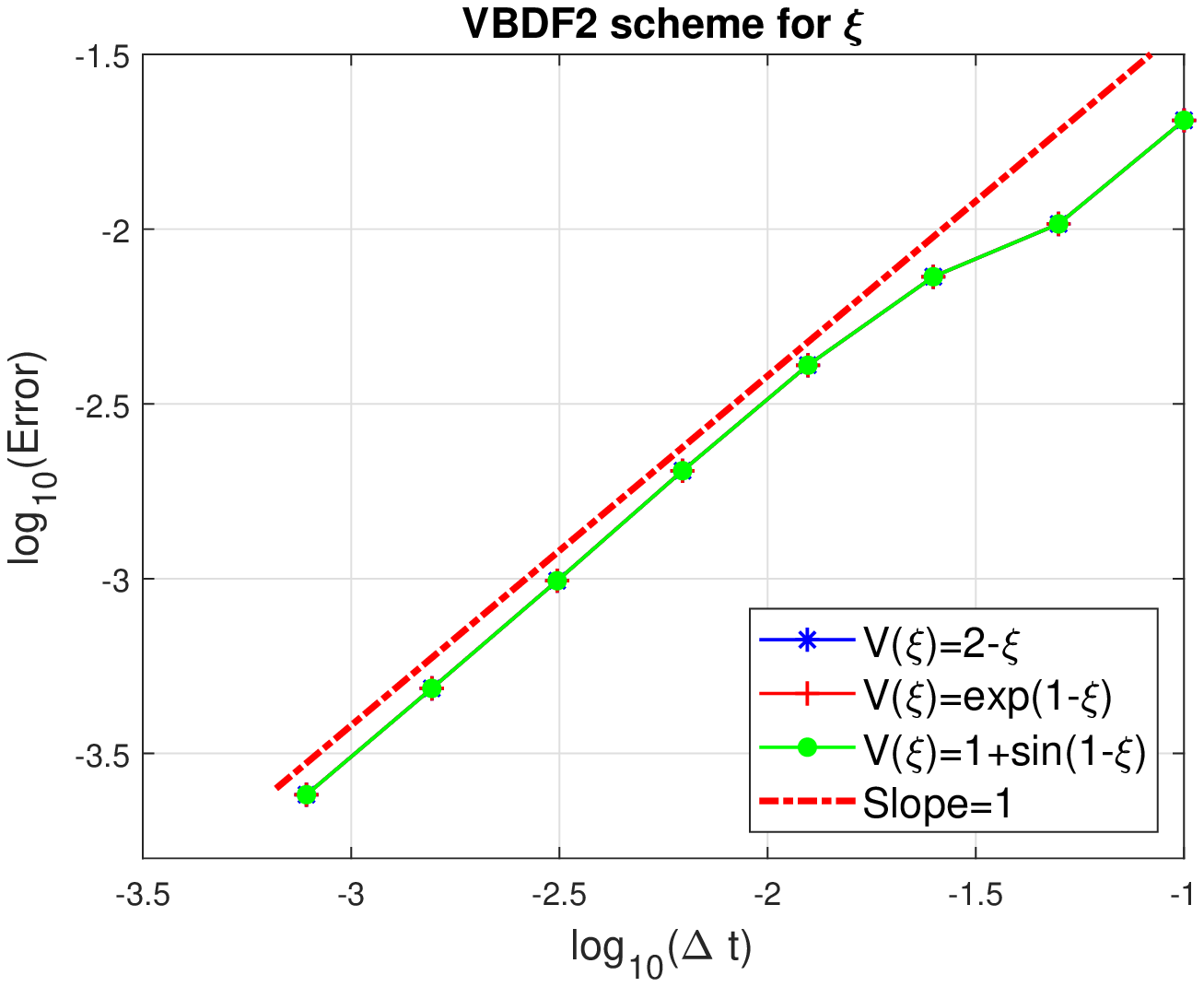}}
\centerline{}
\end{minipage}
\caption{Error decay at $T=1$ versus the time step sizes for the first order scheme and BDF-2 scheme  with $C_{0}=0$.
}\label{fig1}
\end{figure*}

\begin{table}
  \begin{center}
      \caption{ Numerical accuracy of VBDF2 scheme at T=1.}\label{tabel1}
    \begin{tabular}{|ccc|cc|cc|} \hline
    \multicolumn{3}{|c|}{temporal mesh}&\multicolumn{2}{|c|}{$L^{2}$ gradient flow}&\multicolumn{2}{|c|}{$H^{-1}$ gradient flow} \\
     $N$&$\tau$&$\max\{\gamma_{n}\}$&error&order&error&order\\ \hline
     $20$ &7.75e-02&3.14&1.23e-03&   --  &2.01e-01&   --\\
     $40$ &3.63e-02&5.67&2.68e-04&2.00&5.35e-02&1.74\\
     $80$ &2.22e-02&4.08&8.04e-05&2.46&1.86e-02&2.15\\
    $160$ &1.10e-02&4.26&1.25e-05&2.67&5.46e-03&1.76\\
    $320$ &5.38e-03&6.78&3.14e-06&1.92&1.54e-03&1.77\\
    $640$ &2.75e-03&5.63&9.53e-07&1.78&4.47e-04&1.85\\
   $1280$ &1.39e-03&6.47&2.45e-07&1.99&1.17e-04&1.96\\
   $2560$ &6.96e-04&7.72&5.47e-08&2.16&2.79e-05&2.06\\
      \hline
    \end{tabular}
  \end{center}
\end{table}
\subsection{Coarsening dynamics}

Finally we consider an application of the proposed approach with adaptive time-stepping strategy to investigate the coarsening process governed by the Cahn-Hilliard equation \eqref{ex2}, subject to the periodic boundary condition with the computational domain $\Omega:=(0,2\pi)^{2}$. Here, we set $\varepsilon^{2}=0.01$ and the initial condition $\phi(\x,0)$ to be a random data from $-0.05$ to $0.05$.

It is known that the process of the spinodal decomposition usually requires a long time simulation,
thus it is desired to apply an adaptive time stepping strategy for the proposed nonuniform BDF-2 scheme \eqref{BDF_2} to improve the computational efficiency.
This will be particularly useful and effective for the case that the free energy varies little in some time intervals and changes fast in some other intervals.
As the proposed numerical scheme \eqref{BDF_2} are unconditional stable on the general temporal meshes with $\gamma_{**}\leq 4.8645.$ Thus it is easy to employ the adaptive time stepping strategy in our BDF-2 scheme \eqref{BDF_2}, in which the small and the large time step sizes can be applied according to the varying rate of the free energy or the solutions.
Actually, there already exist some adaptive strategies employed in the unconditionally stable schemes
for the gradient flows. Here, we will adopt the following robust strategy based on the energy
variation \cite{QZT11}:
\bq\label{adp}
\dps\Dt_{n+1}=\min\Big(\max\big(\Dt_{min},\frac{\Dt_{max}}{\sqrt{1+\gamma|E^{'}(t)|^{2}}}\big),4.8645\cdot\Dt_{n}\Big),
\eq
where $\Dt_{min},\Dt_{max}$ are predetermined minimum and maximum time step sizes,
$\gamma$ is a constant to be determined.
Obviously according to this strategy, the scheme will automatically select small time step sizes when the energy variation is big and large time step sizes when the change of energy is small, seeing Figure \ref{fig2_2}. Next, we will investigate the efficiency of the BDF-2 scheme \eqref{BDF_2} with this kind of adaptive strategy.

The simulation is performed by using BDF-2 scheme \eqref{BDF_2} in time and Fourier spectral method in space with $128\times128$ Fourier modes.
In Figure \ref{fig2_1}, it displays a comparison on the solution snapshot evolution
between uniform large time step sizes,
adaptive stepping, and fixed small time step sizes up to $T=20$. It is observed that
there is no distinguishable difference at early time (at about $t=0.1$)
as the energy varies slowly at this stage.
After the energy undergoing large variation, the large time step size $\Dt=10^{-2}$ yields
inaccurate $\phi$, while
the adaptive time strategy gives the correct coarsening pattern which is consistent
with the results by the small time step case $\Dt=10^{-4}$.
In Figure \ref{fig2_2} (a), we plot the evolutions of the free energy in time with three different types of temporal meshes.
It has shown the energy dissipation of the proposed VBDF2 scheme consists very well with the one using the fixed small uniform temporal mesh.
The evolution of the adaptive time step sizes in time are displayed in Figure \ref{fig2_2},
which indicates the efficiency of the proposed VBDF2 scheme combining with the adaptive strategy \eqref{adp}.

\begin{figure*}[htbp]
\begin{minipage}[t]{0.19\linewidth}
\centerline{\includegraphics[scale=0.22]{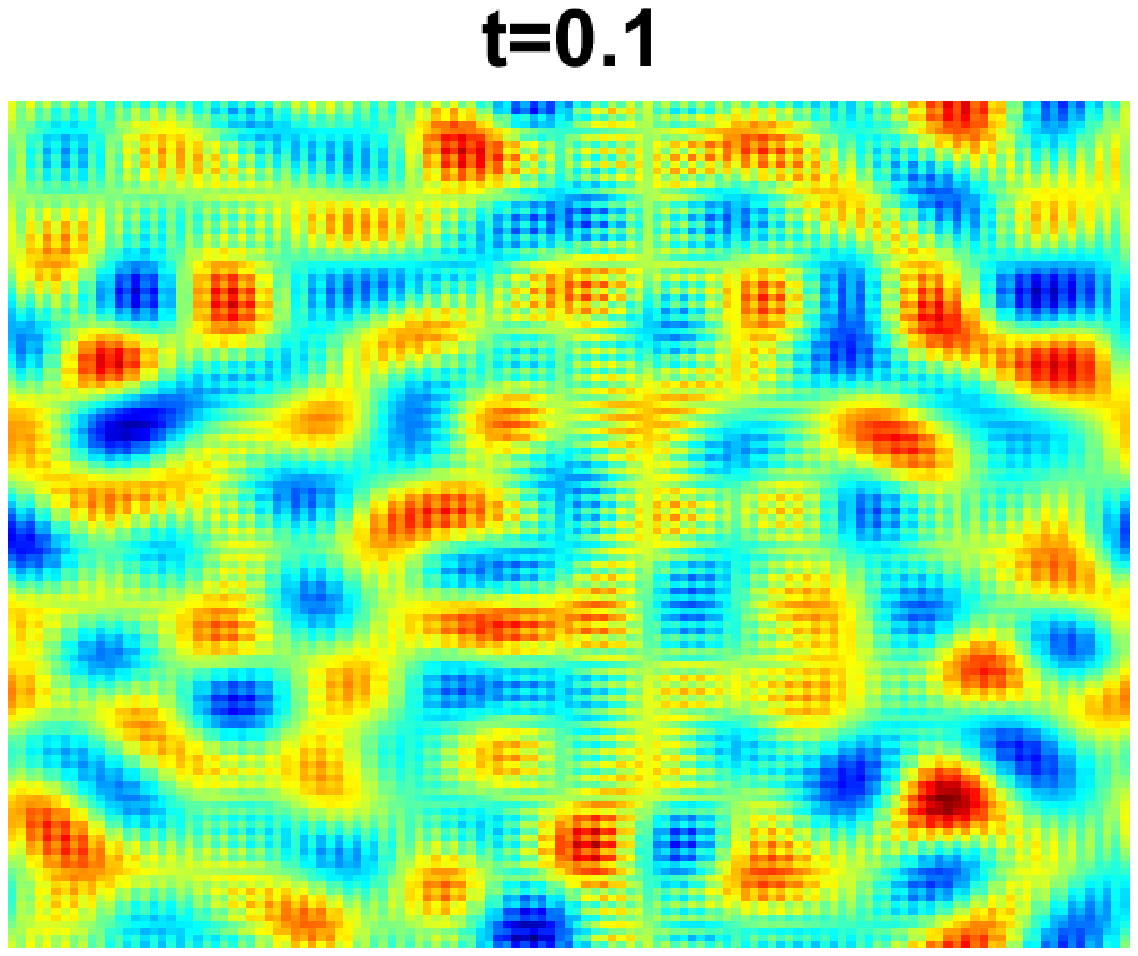}}
\centerline{}
\end{minipage}
\begin{minipage}[t]{0.19\linewidth}
\centerline{\includegraphics[scale=0.22]{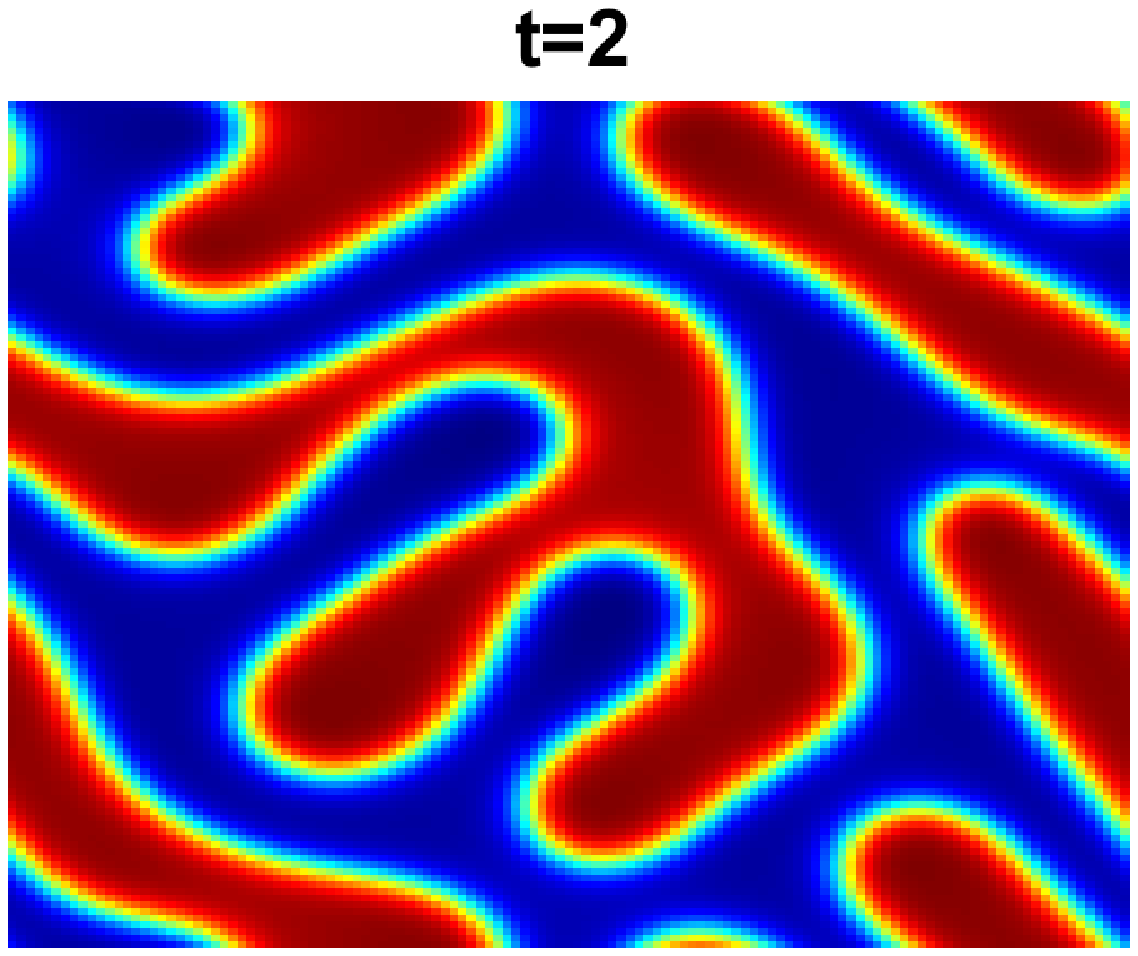}}
\centerline{}
\end{minipage}
\begin{minipage}[t]{0.19\linewidth}
\centerline{\includegraphics[scale=0.22]{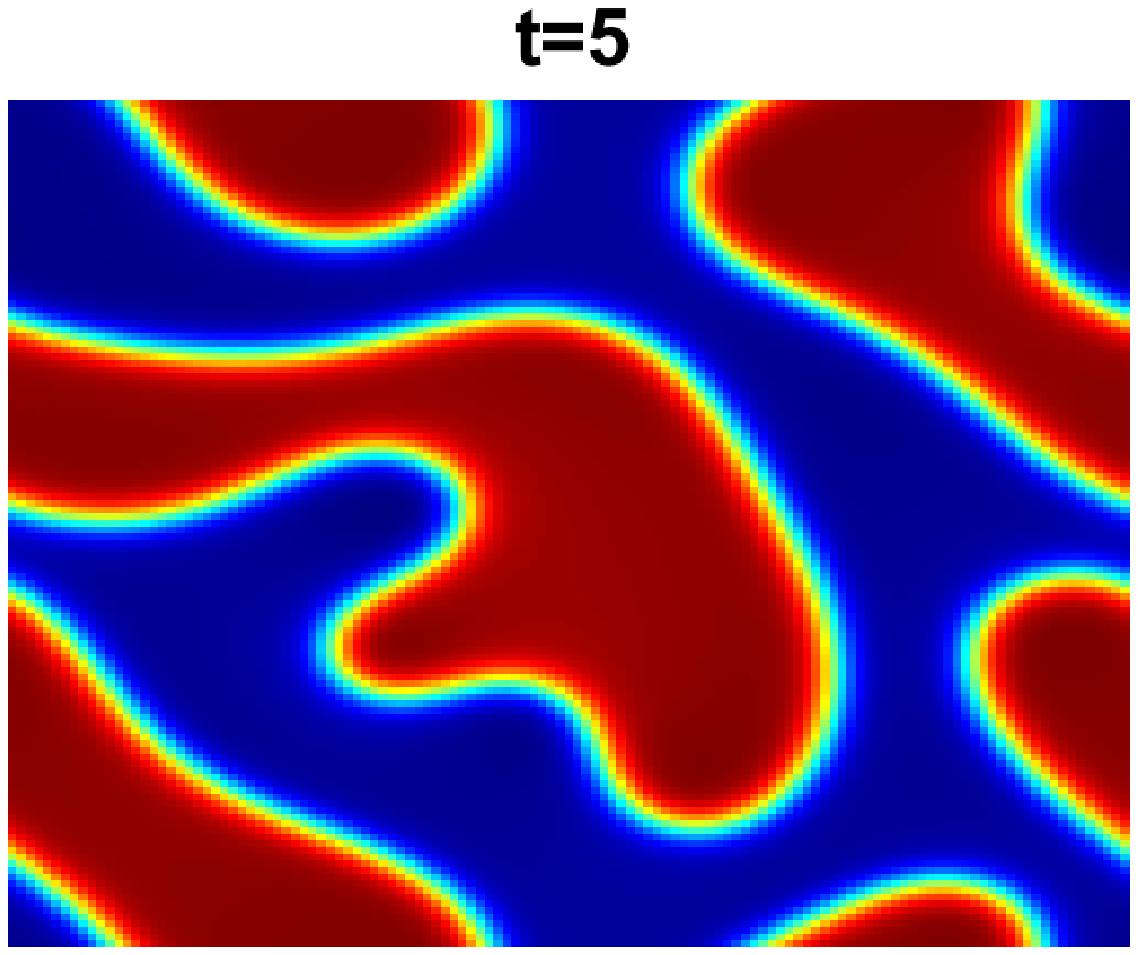}}
\centerline{(a) fixed time step size $\tau=0.01$}
\end{minipage}
\begin{minipage}[t]{0.19\linewidth}
\centerline{\includegraphics[scale=0.22]{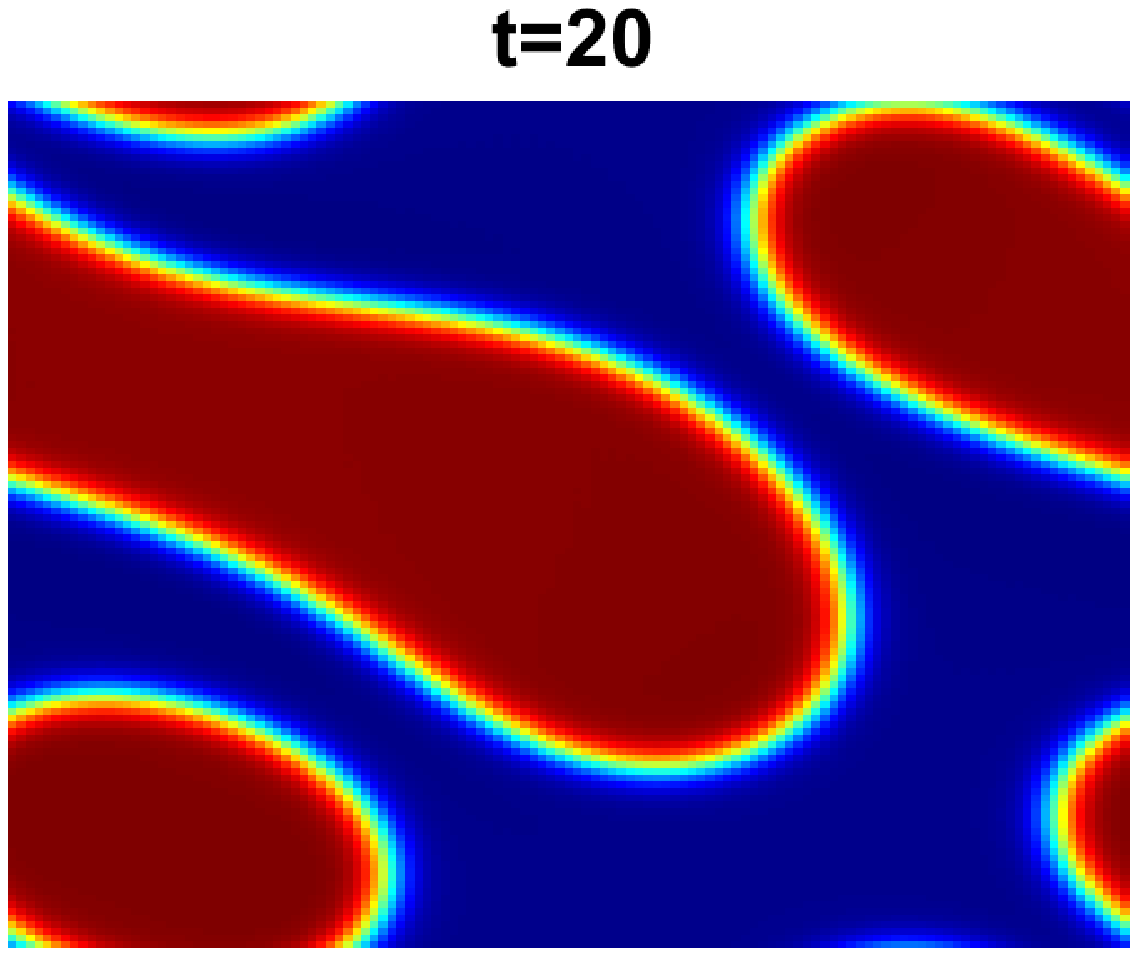}}
\centerline{}
\end{minipage}
\begin{minipage}[t]{0.19\linewidth}
\centerline{\includegraphics[scale=0.22]{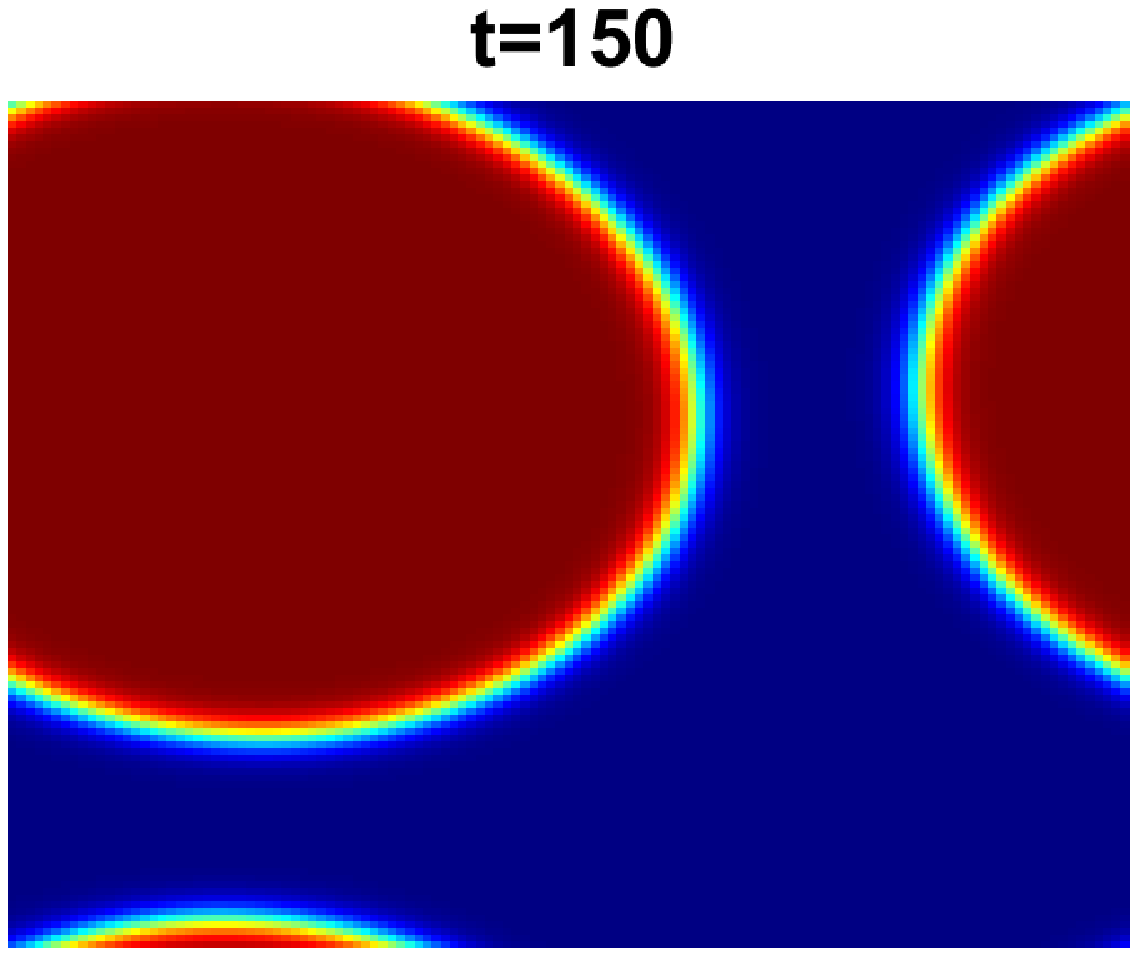}}
\centerline{}
\end{minipage}
\vskip 1mm
\begin{minipage}[t]{0.19\linewidth}
\centerline{\includegraphics[scale=0.22]{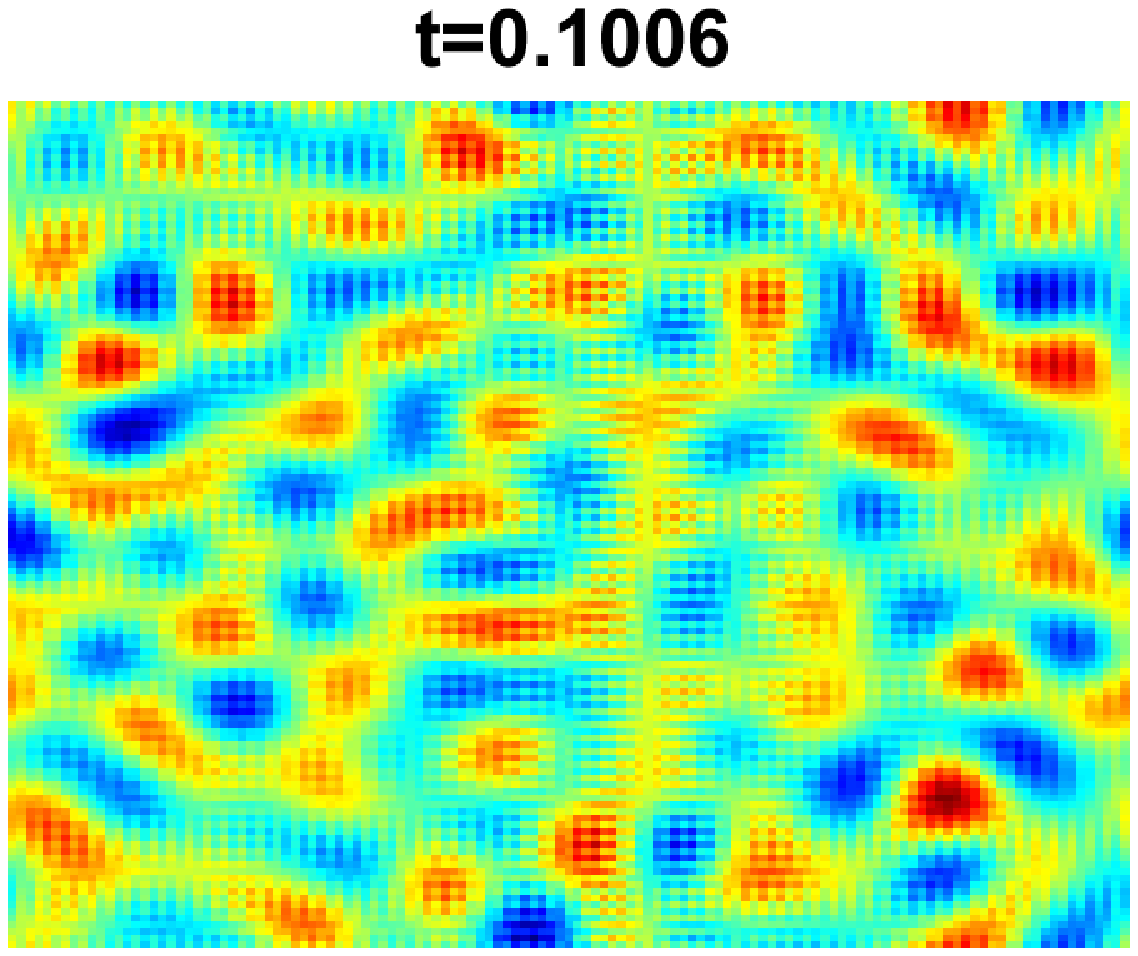}}
\centerline{}
\end{minipage}
\begin{minipage}[t]{0.19\linewidth}
\centerline{\includegraphics[scale=0.22]{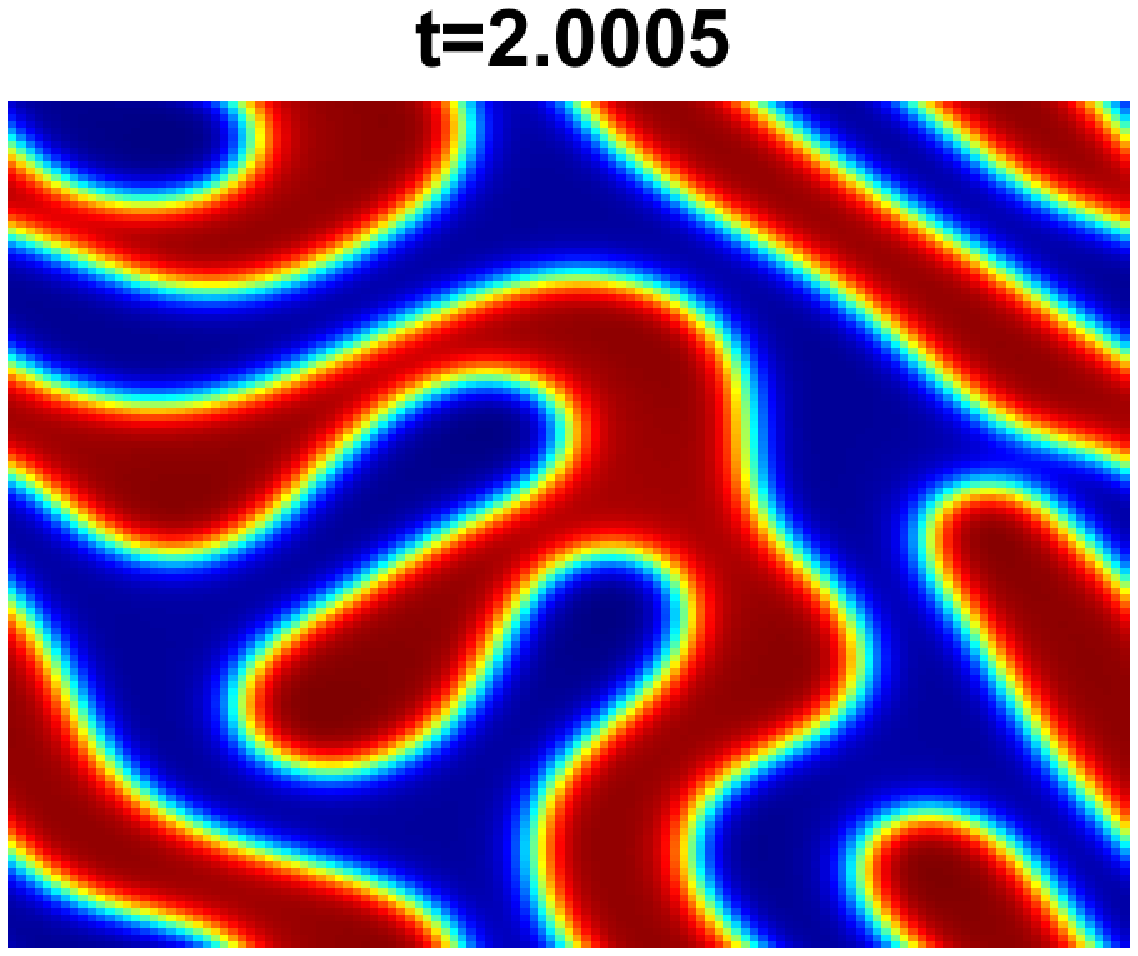}}
\centerline{}
\end{minipage}
\begin{minipage}[t]{0.19\linewidth}
\centerline{\includegraphics[scale=0.22]{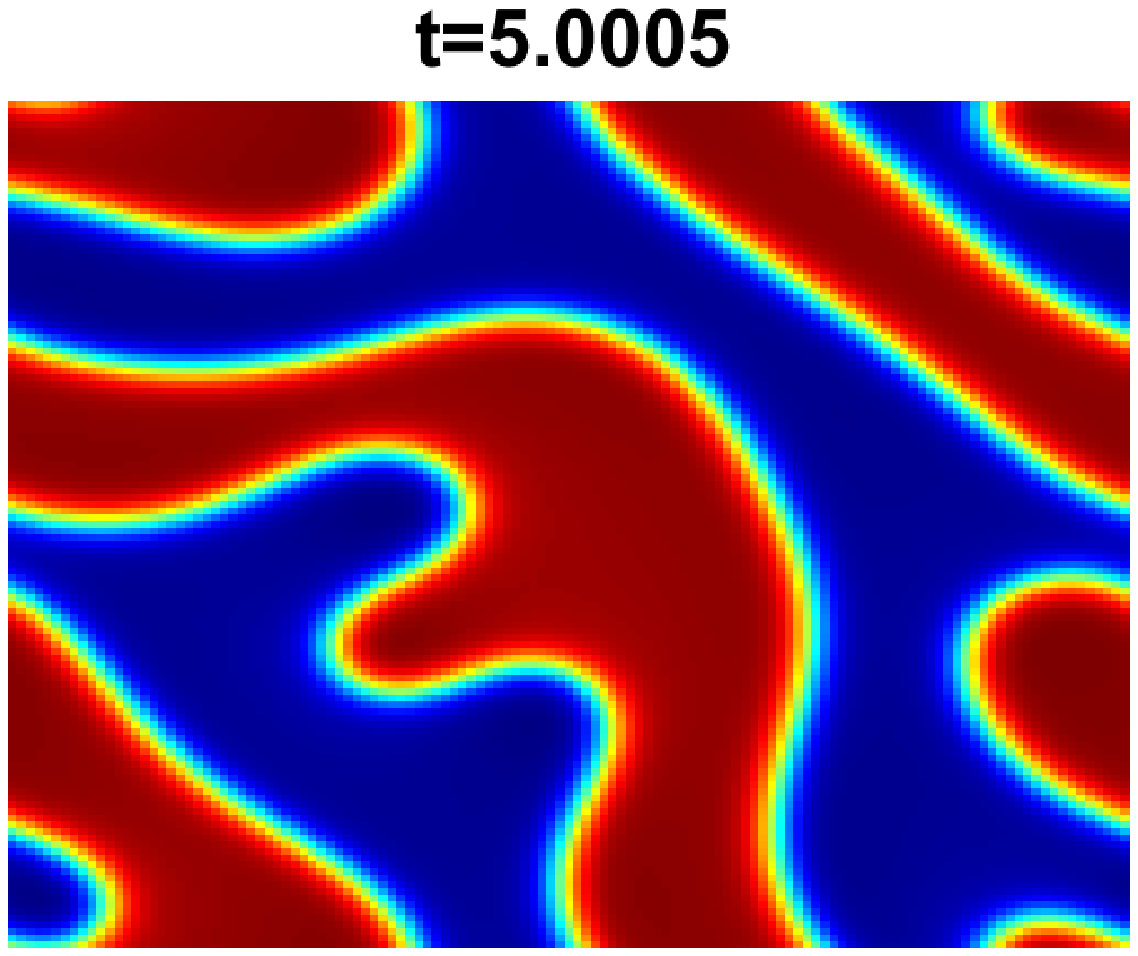}}
\centerline{(b) Adaptive time steps  with $\tau_{min}=10^{-4}, \tau_{max}=10^{-2}$ and $\gamma=1000.$}
\end{minipage}
\begin{minipage}[t]{0.19\linewidth}
\centerline{\includegraphics[scale=0.22]{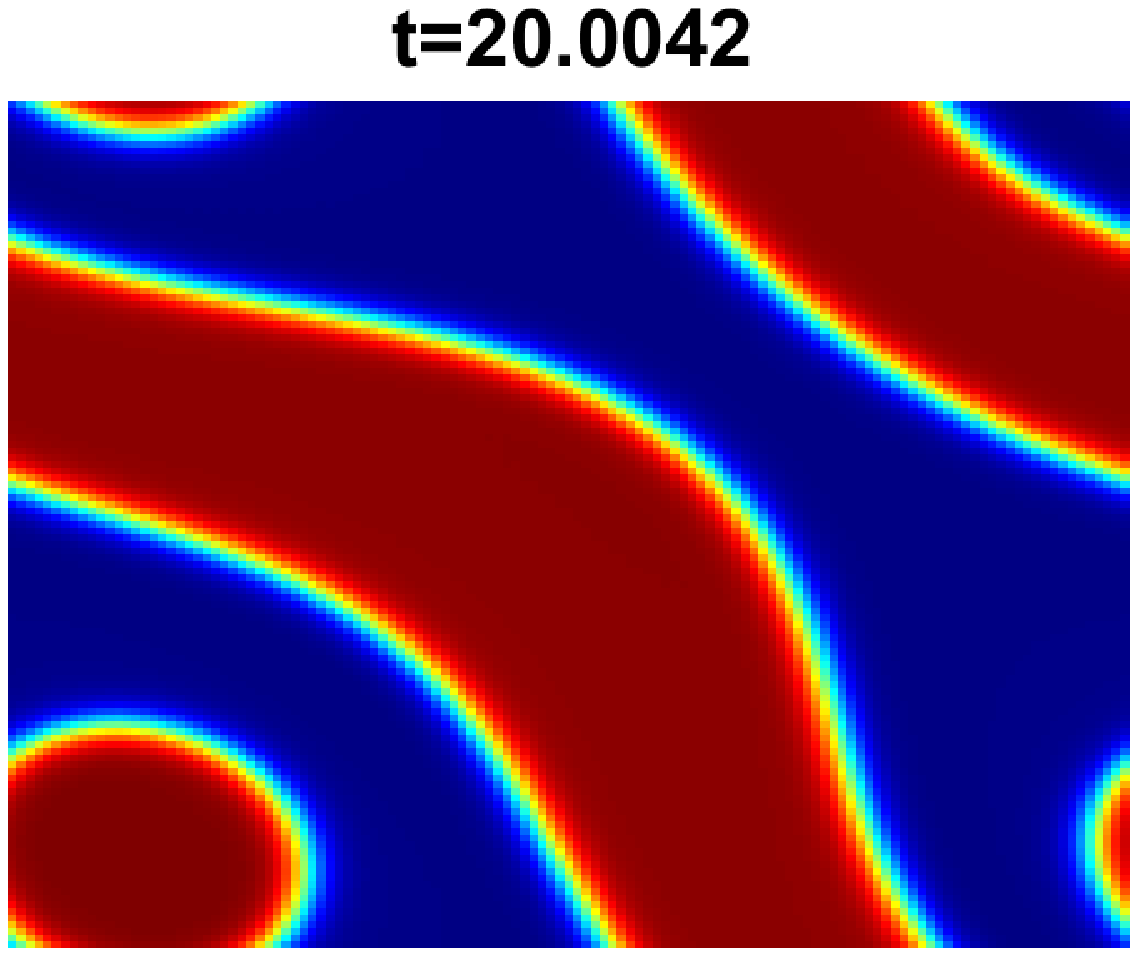}}
\centerline{}
\end{minipage}
\begin{minipage}[t]{0.19\linewidth}
\centerline{\includegraphics[scale=0.22]{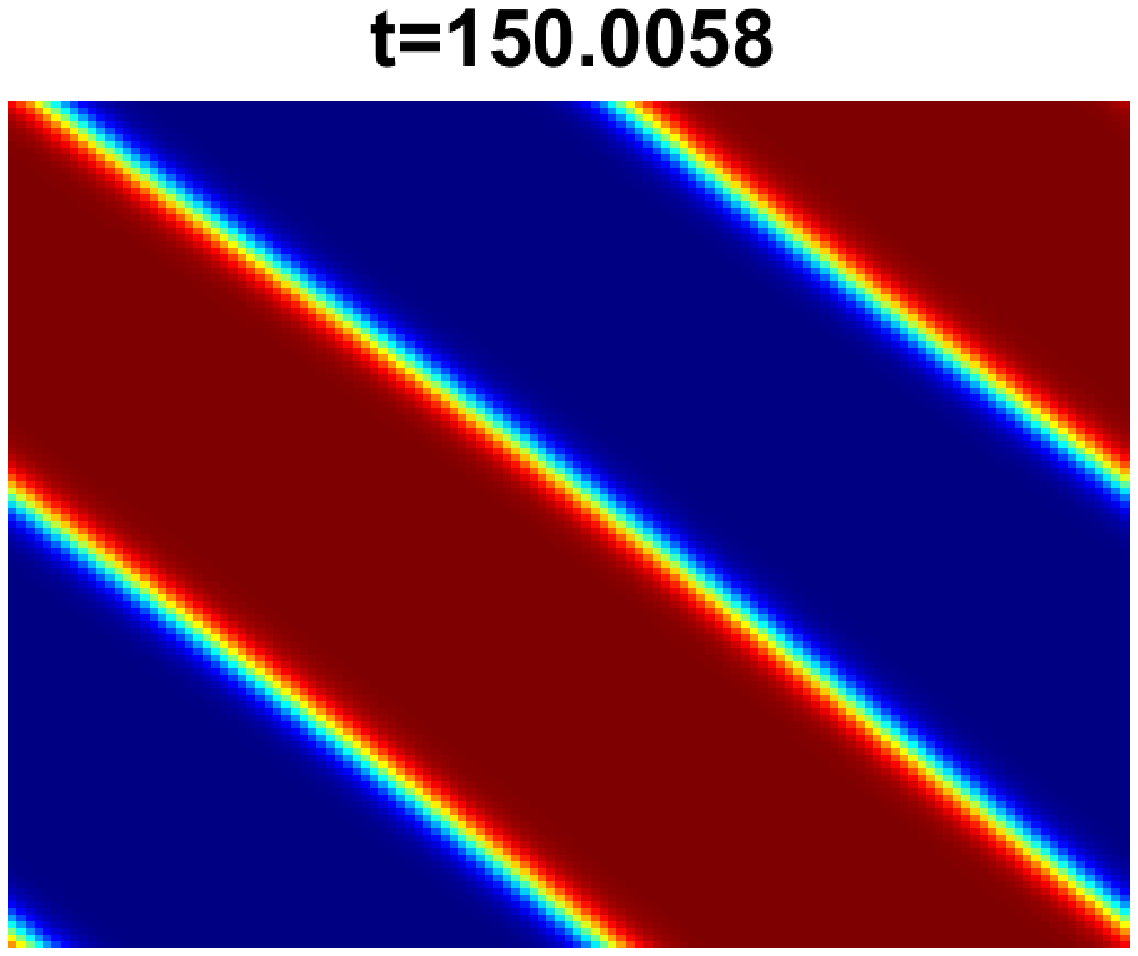}}
\centerline{}
\end{minipage}
\vskip 1mm
\begin{minipage}[t]{0.19\linewidth}
\centerline{\includegraphics[scale=0.22]{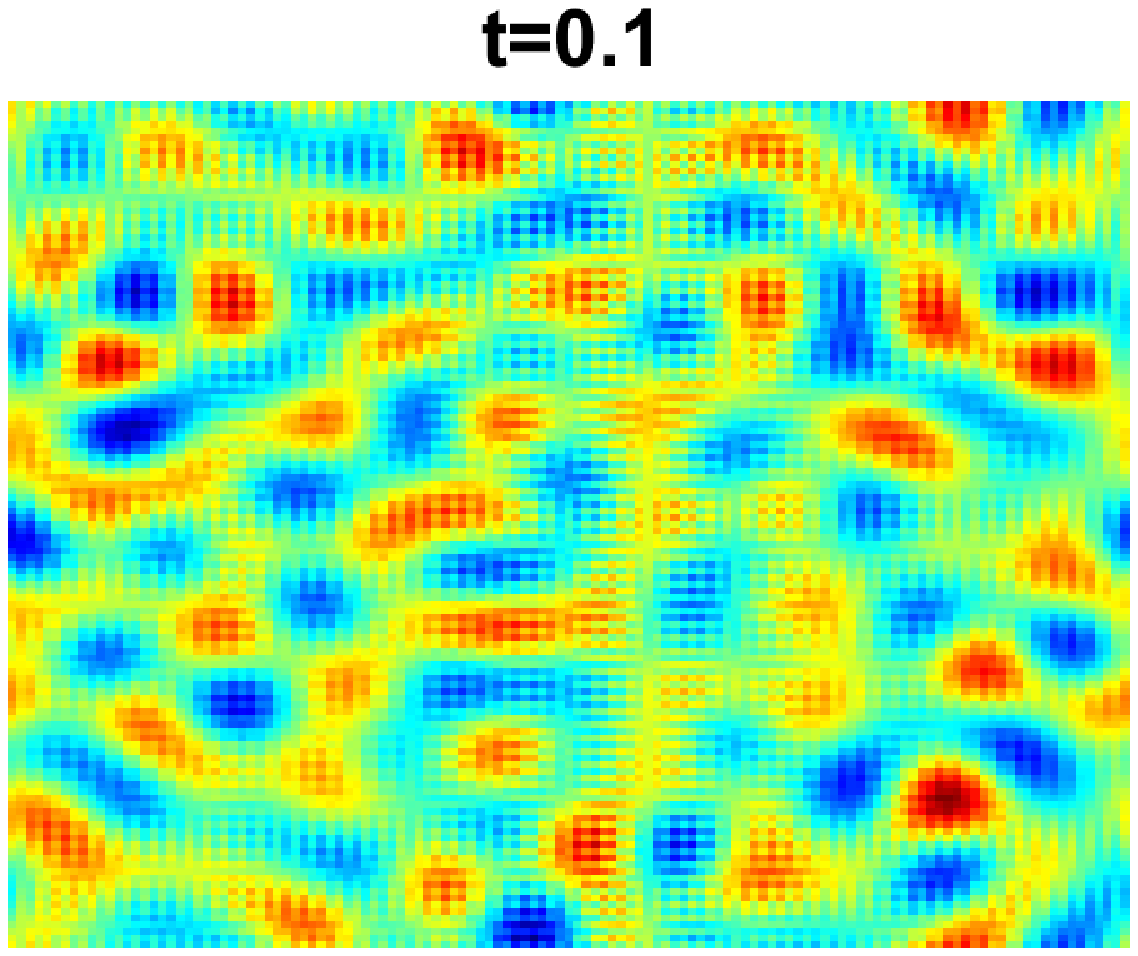}}
\centerline{}
\end{minipage}
\begin{minipage}[t]{0.19\linewidth}
\centerline{\includegraphics[scale=0.22]{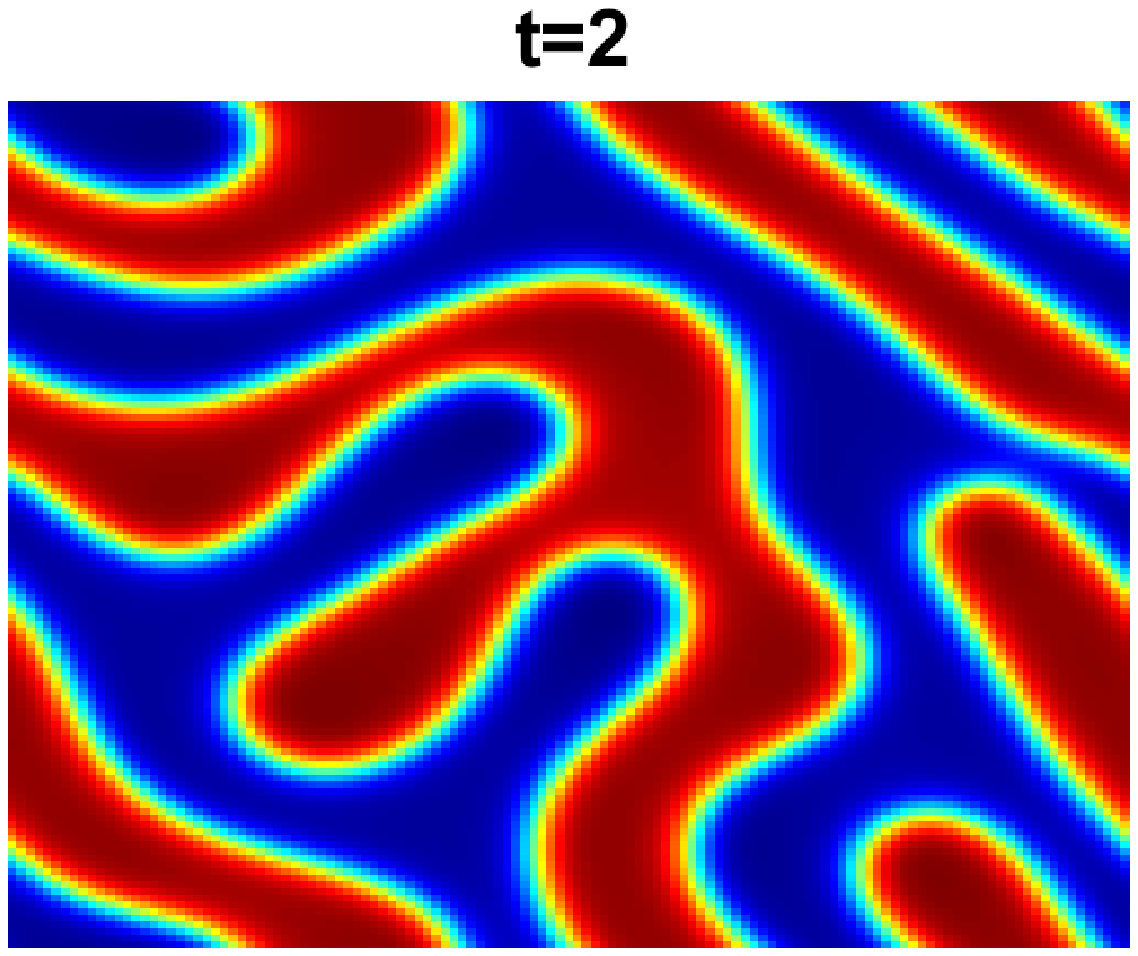}}
\centerline{}
\end{minipage}
\begin{minipage}[t]{0.19\linewidth}
\centerline{\includegraphics[scale=0.22]{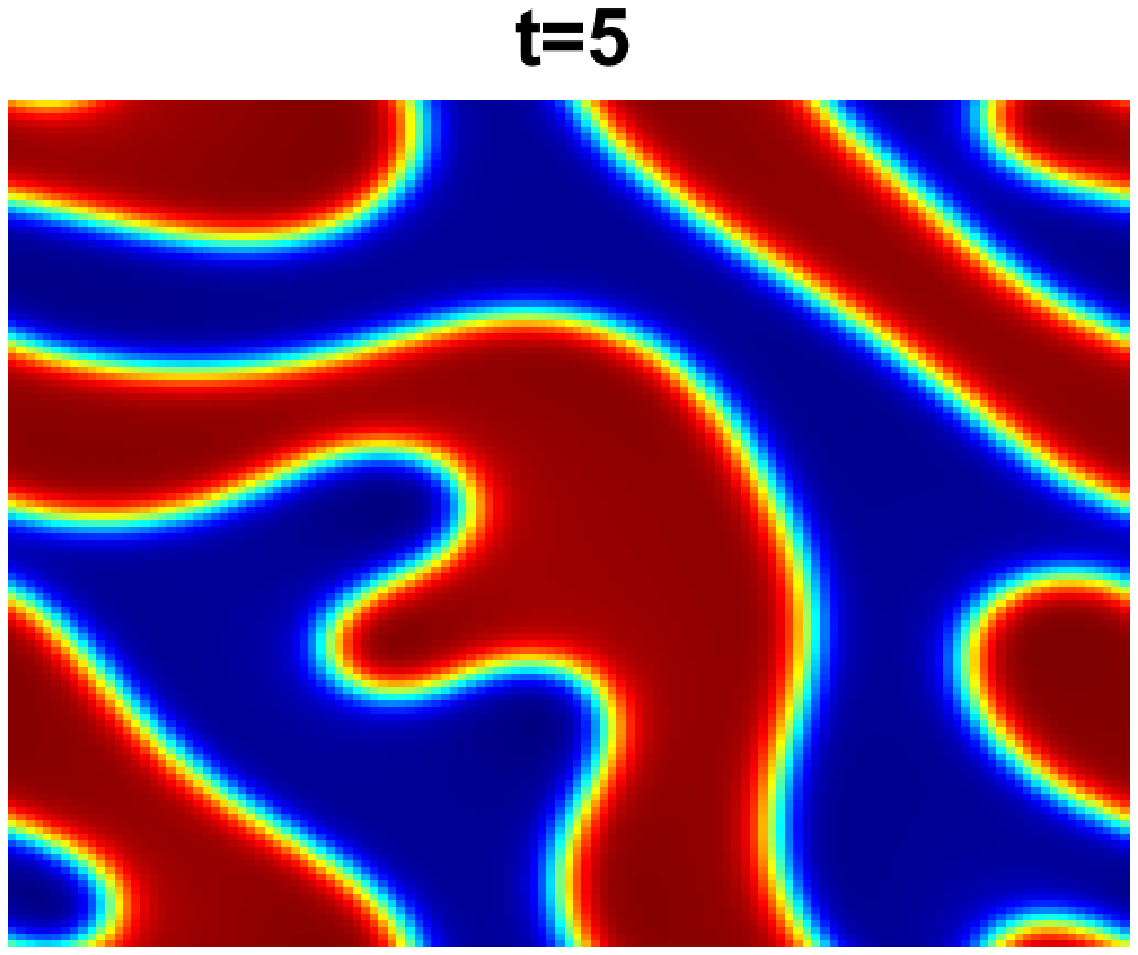}}
\centerline{(c) fixed time step size $\tau=10^{-4}$}
\end{minipage}
\begin{minipage}[t]{0.19\linewidth}
\centerline{\includegraphics[scale=0.22]{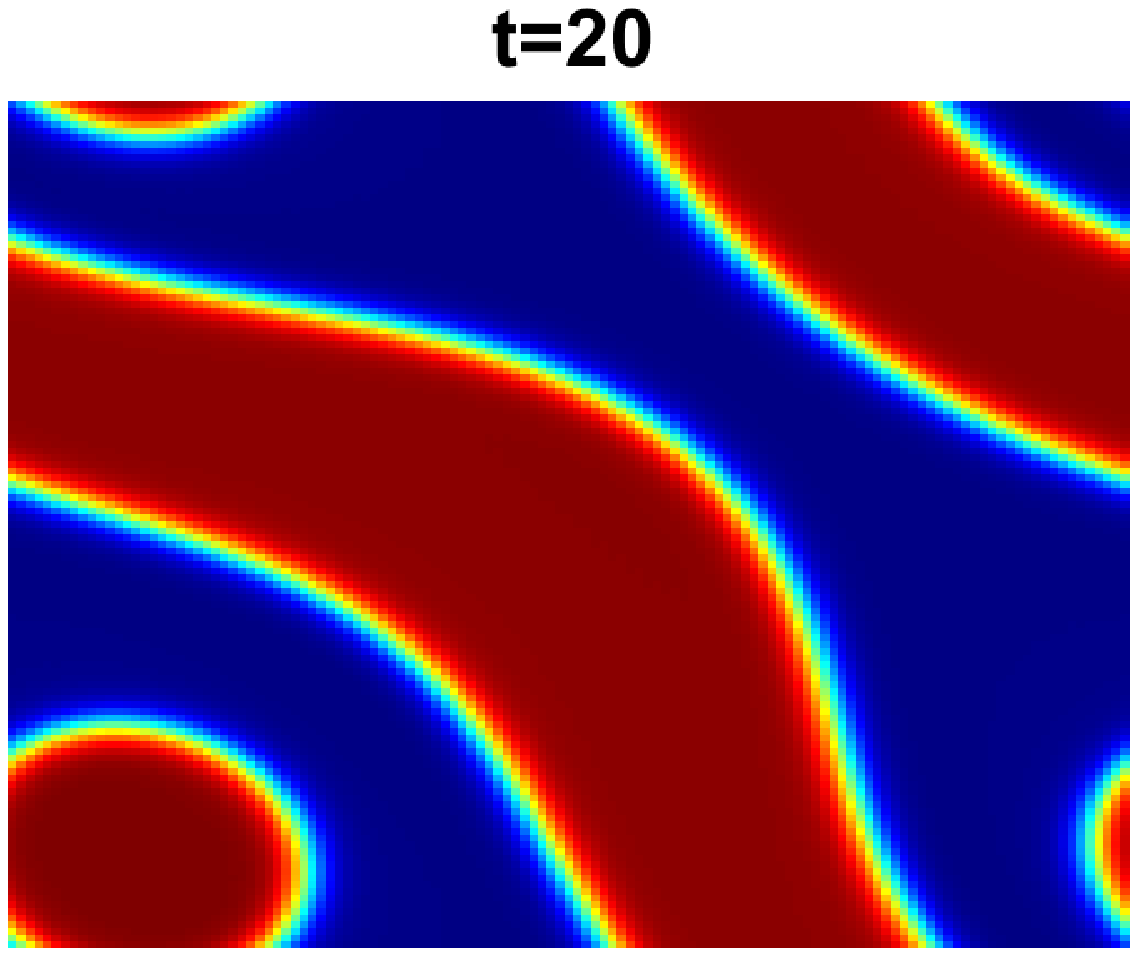}}
\centerline{}
\end{minipage}
\begin{minipage}[t]{0.19\linewidth}
\centerline{\includegraphics[scale=0.22]{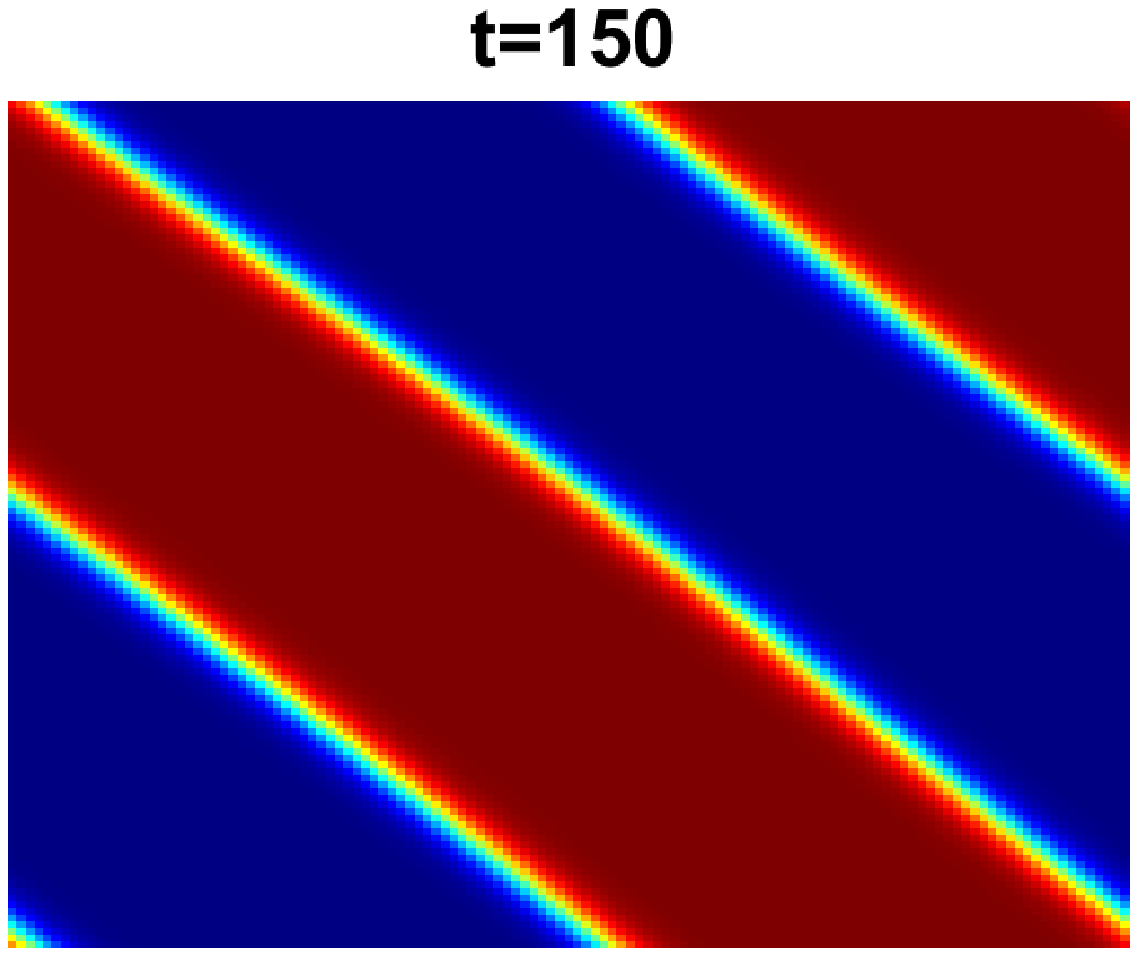}}
\centerline{}
\end{minipage}
\caption{Solution snapshots of coarsening dynamics for the Cahn-Hilliard equation at $t=0.1, 2, 5, 20, 150$, reseptively.
}\label{fig2_1}
\end{figure*}

\begin{figure*}[htbp]
\begin{minipage}[t]{0.49\linewidth}
\centerline{\includegraphics[scale=0.45]{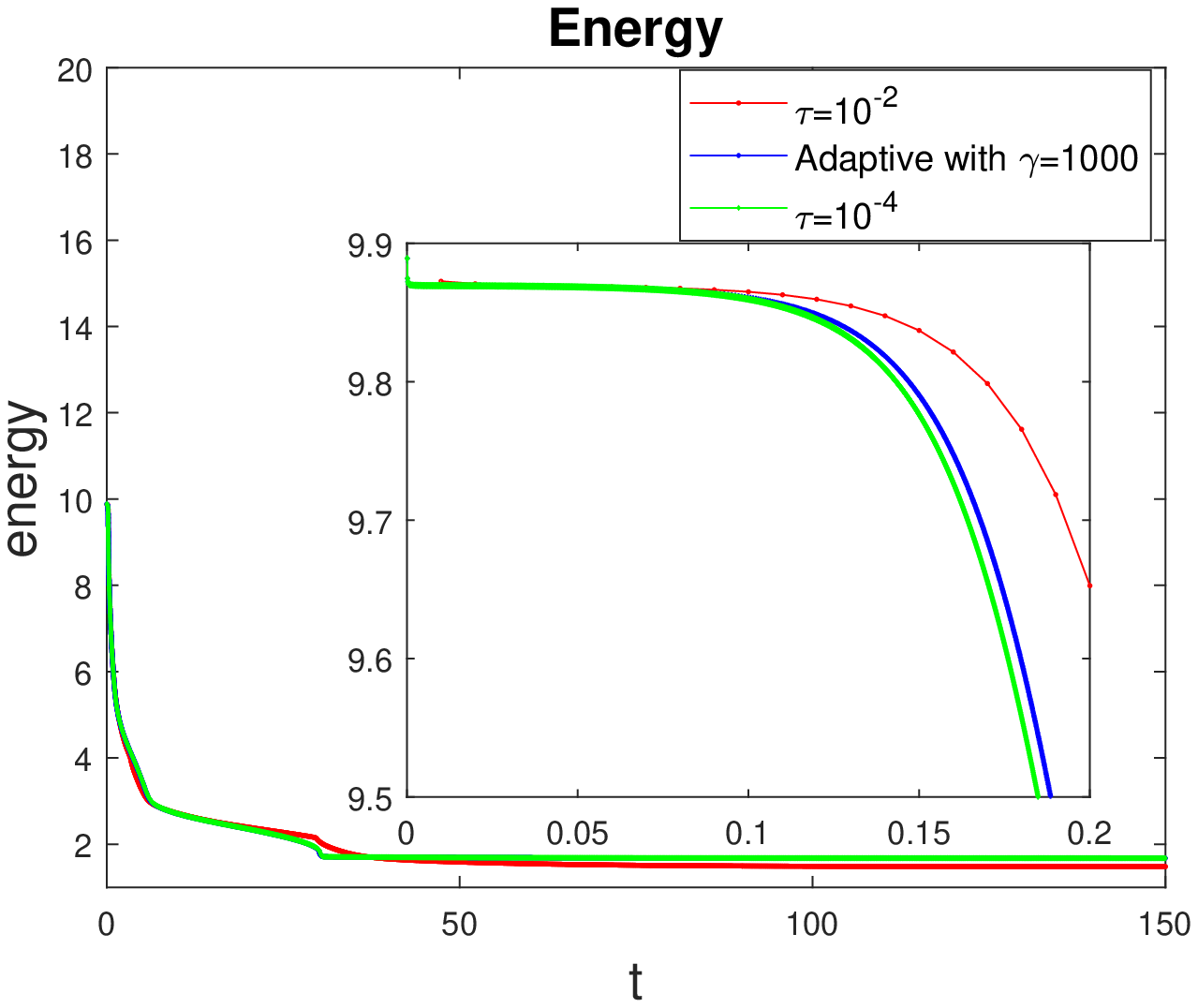}}
\centerline{(a) energy evolution in time}
\end{minipage}
\begin{minipage}[t]{0.49\linewidth}
\centerline{\includegraphics[scale=0.45]{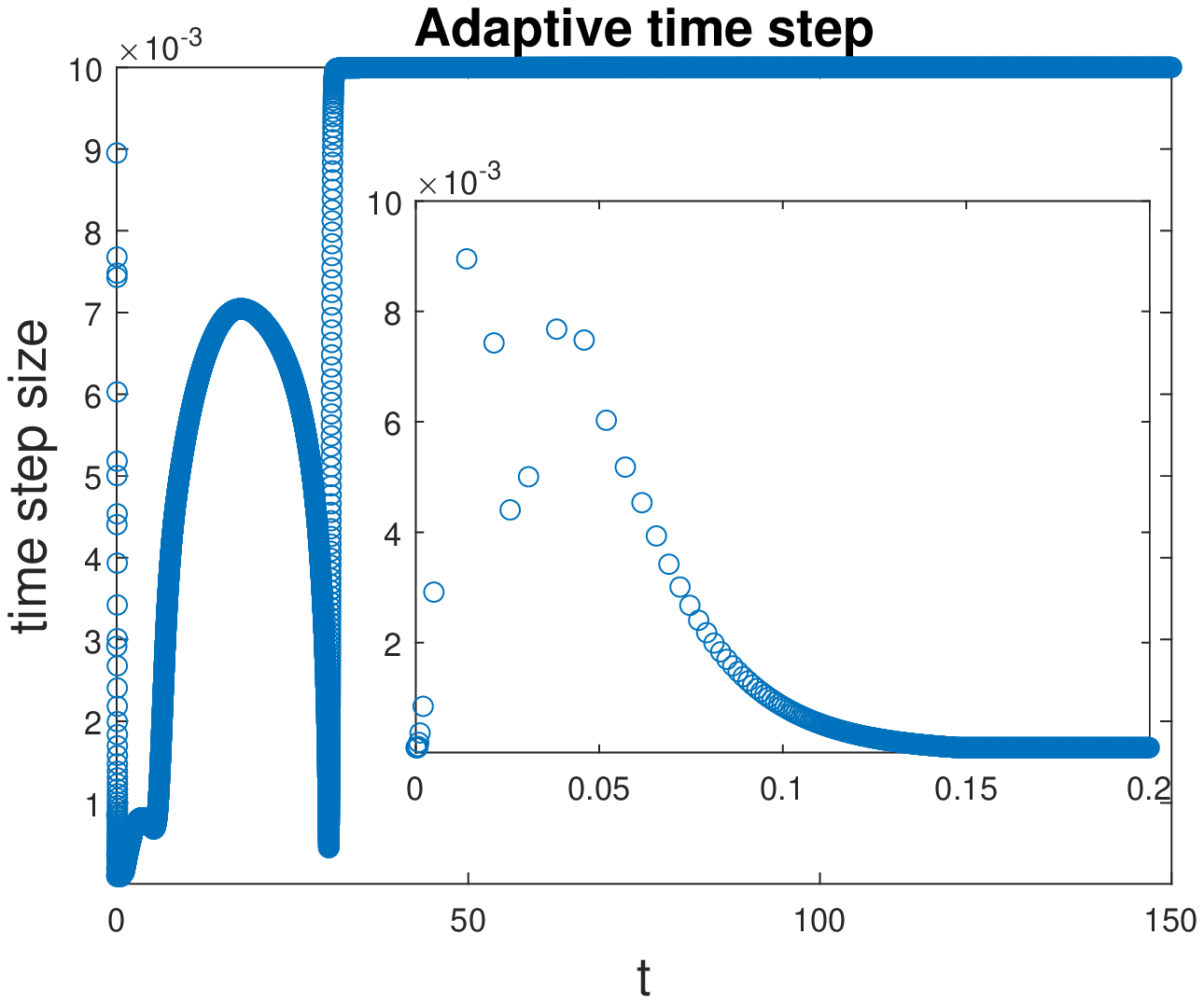}}
\centerline{(b) adaptive time step sizes with $\gamma=1000$}
\end{minipage}
\caption{Evolution in time of the energy and the adaptive time step sizes.
}\label{fig2_2}
\end{figure*}

\section {Concluding remarks}\label{sec:sect6}
We have proposed a variant of the scalar auxiliary variable approach for gradient flows with a technique achieving first order approximation to the auxiliary variable without affecting the second order accuracy of the unknown phase function $\phi$.
Starting with this new approach, we have constructed a second order BDF scheme for gradient flows on the nonuniform temporal mesh, in which the existing adaptive time stepping strategies can be easily adopted.
Moreover, the associated stability and error estimate of the proposed VBDF2 scheme have been rigorously established with a mild assumption on the regularity of the solution and the requirement on the adjacent time step ratios $\gamma_{n+1}\leq4.8645.$
Finally, a series of  numerical experiments have been carried out to validate the theoretical claims.

%


\bibliographystyle{siamplain}
\bibliography{references}
\end{document}